\documentclass[onefignum,onetabnum]{siamonline220329}

\usepackage{lipsum}
\usepackage{amsfonts}
\usepackage{graphicx}
\usepackage{epstopdf}
\usepackage{algorithmic}
\usepackage{mathtools}
\usepackage{mathrsfs}
\usepackage{amssymb}
\usepackage{booktabs}
\usepackage{multirow}
\ifpdf
  \DeclareGraphicsExtensions{.eps,.pdf,.png,.jpg}
\else
  \DeclareGraphicsExtensions{.eps}
\fi

\newcommand{\RR}{{\mathbb R}}

\newcommand{\CC}{{\mathbb C}}
\newcommand{\Z}{{\mathbb Z}}
\newcommand{\N}{\mathbb{N}}
\newcommand{\bx}{\mathbf{x}}
\newcommand{\ca}{\mathbf{a}}
\newcommand{\by}{\mathbf{y}}
\newcommand{\bz}{\mathbf{z}}
\newcommand{\K}{\mathbf{K}}

\newcommand{\bA}{\mathcal{A}}

\newcommand{\bR}{\mathcal{R}}

\newcommand{\bg}{\mathbf{g}}

\newcommand{\br}{\mathbf{r}}

\newcommand{\bu}{\mathbf{u}}
\newcommand{\bi}{\mathbf{i}}

\def\ba{{\boldsymbol{\alpha}}}
\def\bb{{\boldsymbol{\beta}}}
\def\bbg{{\boldsymbol{\gamma}}}

\newcommand{\ud}{\mathrm{d}}

\newcommand{\supp}[1]{\mbox{\upshape supp}(#1)}

\newcommand{\re}{\mathbb{R}}

\newcommand{\Qk}{\mathcal{Q}_k}

\newcommand{\bM}{\mathbf{M}}

\newif\ifcomment
\commentfalse
\commenttrue

\usepackage{enumitem}
\setlist[enumerate]{leftmargin=.5in}
\setlist[itemize]{leftmargin=.5in}

\newsiamremark{remark}{Remark}
\newsiamremark{hypothesis}{Hypothesis}
\crefname{hypothesis}{Hypothesis}{Hypotheses}
\newsiamthm{claim}{Claim}
\newsiamremark{assumption}{Assumption}
\newsiamremark{example}{Example}

\headers{Exploiting Sign Symmetries in Minimizing SORF}{F. Guo, J. Wang, and J. Zheng}

\title{Exploiting Sign Symmetries in Minimizing Sums of Rational Functions\thanks{Submitted to the editors DATE.
\funding{This work was funded by National Key R\&D Program of China under grant No. 2023YFA1009401, Natural Science Foundation of China under grant No. 12201618 and 12171324.}}}

\author{Feng Guo\thanks{School of Mathematical Sciences, Dalian University of Technology, Dalian 116024, Liaoning Province, China (\email{fguo@dlut.edu.cn}).}
\and Jie Wang\thanks{Academy of Mathematics and Systems Science, Chinese academy of sciences, Beijing 100190, China (\email{wangjie212@amss.ac.cn}).}
\and Jianhao Zheng\thanks{School of Mathematical Sciences, Dalian University of Technology, Dalian 116024, Liaoning Province, China.}}

\usepackage{amsopn}

\ifpdf
\hypersetup{
  pdftitle={An Example Article},
  pdfauthor={D. Doe, P. T. Frank, and J. E. Smith}
}
\fi

\begin{document}

\maketitle

\begin{abstract}
This paper is devoted to the problem of minimizing a sum of rational functions over a basic semialgebraic set. We provide a hierarchy of sum of squares (SOS) relaxations that is dual to the generalized moment problem approach due to Bugarin, Henrion, and Lasserre. The investigation of the dual SOS aspect offers two benefits: 1) it allows us to conduct a convergence rate analysis for the hierarchy; 2) it leads to a sign symmetry adapted hierarchy consisting of block-diagonal semidefinite relaxations. When the problem possesses correlative sparsity as well as sign symmetries, we propose sparse semidefinite relaxations by exploiting both structures. Various numerical experiments are performed to demonstrate the efficiency of our approach. Finally, an application to maximizing sums of generalized Rayleigh quotients is presented.
\end{abstract}



\section{Introduction}
In this paper, we consider the optimization problem of minimizing a sum of rational functions:
\begin{equation}\label{eq::pro}
\rho\coloneqq\inf_{x\in\K} \sum_{i=1}^N \frac{p_i(\bx)}{q_i(\bx)},
\tag{SRFO}
\end{equation}
where 
\begin{equation}\label{eq::K}
\K\coloneqq\{\bx\in\RR^n \mid g_j(\bx)\ge 0,\ j=1,\ldots,m\},
\end{equation}
and $p_i, q_i, g_j$ are polynomials in variables $\bx=(x_1,\ldots,x_n)$. Throughout the paper, we make the following assumption on \eqref{eq::pro}:
\begin{assumption}\label{as::0}
(i) $\K$ is compact; (ii) $q_i>0$ on $\K$ for $i=1,\ldots,N$.	
\end{assumption}
The problem \eqref{eq::pro} has applications in various fields, including computer vision \cite{hartley2013verifying,kahl2008practical}, multi-user MIMO systems \cite{PSS2006}, sparse Fisher discriminant analysis in pattern recognition \cite{DFBSR,FN,WZWCL}. 

Given the potentially large number of terms $N$ and the absence of convexity (resp. concavity) assumptions on $p_i$ (resp. $q_i$), globally solving \eqref{eq::pro} or achieving a close approximation of the optimum $\rho$ presents significant challenges. 
In literature, \eqref{eq::pro} is often called a {\itshape sum-of-ratios program} which stands as one of the most complex fractional programs encountered to date. For strategies of tackling sum-of-ratios programs under specific assumptions regarding concavity and linearity of ratios, the reader may refer to the survey \cite{SS2003} and \cite[Table 1]{WSB2008}.

In the scenario where $N=1$, by utilizing the polynomial structure in \eqref{eq::pro} and Positivstellens\"atze from real algebraic geometry, some hierarchies of semidefinite programming (SDP) relaxations were proposed in \cite{Jibetean06, Nie2008}. The basic idea is to maximizing a number $\gamma\in\RR$ subject to the nonnegativity of $p_1-\gamma q_1$ on $\K$ which can be ensured by certain sum of squares (SOS) representations. In \cite{GW2014}, the problem of minimizing a rational function was reformulated as a polynomial optimization problem and solved by the exact Jacobian SDP relaxation method proposed by Nie \cite{nie2013exact}.

For the case that $N>1$, one could attempt to reduce all rational functions $p_i/q_i$ 
to the same denominator and apply the hierarchies of SDP relaxations mentioned above.
However, due to the potentially large value of $N$, the resulting unified denominator 
may have a significantly high degree, which makes it impractical to even solve the 
first order relaxations of the hierarchies.
To overcome this difficulty, Bugarin, Henrion, and Lasserre \cite{BHL2016} reformulated \eqref{eq::pro} as an equivalent infinite-dimensional linear program which is a particular instance of the generalized moment problem (GMP) with $N$ unknown measures. The GMP can be relaxed into a hierarchy of SDPs which provides increasingly tight lower bounds on $\rho$. When a correlative sparsity pattern is present in the polynomial data of \eqref{eq::pro}, a sparse GMP reformulation for \eqref{eq::pro} and a corresponding hierarchy of sparse SDP relaxations were also proposed in \cite{BHL2016}.
By employing pushforward measures, Lasserre et al. \cite{LVMZ2021} gave an approach yielding convergent upper bounds on $\rho$.

\vskip 5pt
\noindent{\bf Contributions.} Our main contributions are summarized as follows:
\begin{enumerate}
\item By studying the Lagrange dual problem, we unveil that the core strategy of the GMP approach is to replace the terms $p_i/q_i$, $i=2,\ldots,N$ by polynomial approximations from below on $\K$, leading to a problem with a single denominator. As a result, we provide a hierarchy of SOS relaxations that is dual to the GMP approach due to Bugarin, Henrion, and Lasserre. 
Moreover, we are able to conduct a convergence rate analysis for the hierarchy of SDP relaxations for \eqref{eq::pro}. 

\item We present a sign symmetry adapted hierarchy consisting of block-diagonal SDP relaxations for \eqref{eq::pro}. Furthermore, when both correlative sparsity and sign symmetries are present in the polynomial data of \eqref{eq::pro}, we show that the exploitation of sign symmetries can be naturally incorporated into the sparse hierarchy of SDP relaxations for \eqref{eq::pro}, thereby further reducing the computational cost. Various numerical experiments demonstrate the efficiency of our approach.
\end{enumerate}

\vskip 7pt
The rest of the paper is organized as follows. 
We first recall some preliminaries and the GMP approach for \eqref{eq::pro}
in Section \ref{sec::pre}. In Section \ref{sec::dual}, we investigate the dual aspect 
of the GMP approach and conduct convergence rate analysis for the SDP hierarchy.
In Section \ref{sec::sparsity}, we develop sparse SDP relaxations for \eqref{eq::pro} 
by exploiting sign symmetries as well as correlative sparsity present in the problem data. Numerical experiments are presented in Section \ref{sec::num} and an application to maximizing sums of generalized Rayleigh quotients is provided in Section \ref{app}. 

\section{Notation and Preliminaries}\label{sec::pre}
Let $\N$ denote the set of nonnegative integers. For $n\in\N\setminus\{0\}$, let $[n]\coloneqq\{1,2,\ldots,n\}$. For $k\in\N$, let $\N^n_k\coloneqq\{\ba=(\alpha_i)\in\N^n\mid \sum_{i=1}^n\alpha_i\le k\}$. We use $|\cdot|$ to denote the cardinality of a set. Let
$\RR[\bx]\coloneqq\RR[x_1,\ldots,x_n]$ be the ring of multivariate polynomials in $n$ variables $\bx$, and $\RR[\bx]_k$ denote the subset of polynomials of degree no greater than $k$. A polynomial $f\in\RR[\bx]$ can be written as $f=\sum_{\ba\in\N^n}f_{\ba}\bx^{\ba}$ with $f_{\ba}\in\RR$ and $\bx^{\ba}\coloneqq x_1^{\alpha_1}\cdots x_n^{\alpha_n}$. The support of $f$ is defined by $\supp{f}\coloneqq\{\ba\in\N^n\mid f_{\ba}\ne0\}$. For $\bA\subseteq\N^n$, $\RR[\bA]$ denotes the set of polynomials with supports contained in $\bA$, i.e., $\RR[\bA]\coloneqq\{f\in\RR[\bx]\mid\supp{f}\subseteq\bA\}$.
For $\bu=(u_1,\ldots,u_n)\in \RR^n$, $\Vert \bu\Vert$ denotes the standard Euclidean norm of $\bu$.
For $t\in\re$, we use $\lceil t\rceil$ to denote the smallest integer that is not smaller than $t$.
We use $A\succeq 0$ to indicate that the matrix $A$ is positive semidefinite. For two matrices $A,B$ of the same size, let $A\circ B$ denote the Hadamard product, defined by $[A\circ B]_{ij} = A_{ij}B_{ij}$.

\subsection{Sums of squares and moments}
We recall some background about SOS polynomials
and the dual theory of {\itshape moment matrices}.
A polynomial $f(\bx)\in\RR[\bx]$ is said to be a
\emph{sum of squares} if it can be written as $f(\bx)=\sum_{i=1}^t
f_i(\bx)^2$ for some $f_1(\bx),\ldots,f_t(\bx)\in\RR[\bx]$. Let
$\Sigma[\bx]$ denote the set of SOS polynomials in $\RR[\bx]$.

Let $\bg\coloneqq\{g_1,\ldots,g_m\}$ be the set of polynomials that defines the
semialgebraic set $\K$ in \eqref{eq::K}. We denote by
\[
  \mathcal{Q}(\bg)\coloneqq\left\{\sigma_0+\sum_{j=1}^m\sigma_jg_j\ \middle|\ 
\sigma_j \in \Sigma[\bx], j\in\{0\}\cup[m] \right\}
\]
the {\itshape quadratic module} generated by $\bg$ and denote by
\[
\Qk(\bg)\coloneqq\left\{\sigma_0+\sum_{j=1}^m\sigma_jg_j\ \middle|\ \sigma_0,\sigma_j \in \Sigma[\bx], \deg(\sigma_0),\deg(\sigma_jg_j)\le 2k, j\in[m]\right\}
\]
the $k$-th {\itshape truncated quadratic module}.
It is clear that if $f\in\mathcal{Q}(\bg)$, then
$f(\bx)\ge 0$ for any $\bx\in \K$ though the converse is not necessarily
true.

Given a (pseudo-moment) sequence of real numbers $\by\coloneqq(y_{\ba})_{\ba\in \mathbb{N}^n}$,
the $k$-th order {\itshape moment matrix} is the matrix $\bM_k(\by)$
indexed by $\mathbb{N}^n_{k}$ with the $(\ba,\bb)$-th entry being $y_{\ba+\bb}$. Given a polynomial $f(\bx)=\sum_{\ba}f_{\ba}\bx^{\ba}$, the $k$-th order {\itshape localizing
matrix} $\bM_{k}(f\by)$ indexed by $\mathbb{N}^n_{k}$ is defined by $[\bM_{k}(f\by)]_{\bb,\bbg}=\sum_{\ba}f_{\ba}y_{\ba+\bb+\bbg}$.
The Riesz functional $L_{\by}$ on $\RR[\bx]$ is defined by $L_{\by}\left(\sum_{\ba}f_{\ba}\bx^{\ba}\right)
\coloneqq\sum_{\ba}f_{\ba}y_{\ba}$.
\begin{definition}[Archimedean condition]\label{def::AC}
We say that $\mathcal{Q}(\bg)$ is {\itshape Archimedean} if
there exists $M>0$ such that $M-x_1^2-\cdots-x_n^2\in\mathcal{Q}(\bg)$.
\end{definition}
Note that the Archimedean condition implies that $\K$ is compact and the converse is not necessarily true. However, for any compact set $\K$, we could always
force the associated quadratic module to be Archimedean
by adding a redundant constraint $M-x_1^2-\cdots-x_n^2\ge 0$ in the description of
$\K$ for sufficiently large $M$.

\begin{theorem}\label{th::PP}{\upshape\cite[{Putinar's Positivstellensatz\/}]{Putinar1993}}
Suppose that $\mathcal{Q}(\bg)$ is Archimedean.
If a polynomial $f\in\RR[\bx]$ is positive on $\K$, then $f\in\Qk(\bg)$ for some $k\in\N$.
\end{theorem}

\subsection{The GMP reformulation and SDP relaxations}
In \cite{BHL2016}, Bugarin et al. reformulated \eqref{eq::pro} as the following GMP:
\begin{equation*}\label{eq::gmp}
\left\{\begin{aligned}
		\inf_{\mu_i\in\mathcal{M}(\K)_+}&\ \sum_{i=1}^N \int_{\K}
		p_i\ud \mu_i\\
		\text{s.t.}\quad\,&\ \int_{\K} q_1\ud\mu_1=1,\\
		&\ \int_{\K}
		\bx^{\ba}q_i\ud\mu_i=\int_{\K}\bx^{\ba}q_1\ud\mu_1,\quad\forall
		\ba\in\N^n, i\in[N]\setminus\{1\},
	\end{aligned}\right.\tag{P}
\end{equation*}
where $\mathcal{M}(\K)_+$ denotes the set of finite positive Borel measures
supported on $\K$.

Let 
\[\begin{aligned}
       d_j&\coloneqq\lceil\deg(g_j)/2\rceil,\quad j\in[m],\\
       d_{\min}&\coloneqq\max\,\{\lceil\deg(p_i)/2\rceil, \lceil\deg(q_i)/2\rceil, i\in [N]; \,d_j, j\in[m]\}.
\end{aligned}\]
Based on the reformulation \eqref{eq::gmp}, Bugarin et al. further proposed the following hierarchy of SDP relaxations for \eqref{eq::pro} ($k\ge d_{\min}$):
\begin{equation*}\label{eq::momentsdp}
\left\{
\begin{aligned}
\inf_{\by_i}&\ \sum_{i=1}^N L_{\by_{i}}(p_i)\\
\text{s.t.}&\ \bM_k(\by_i)\succeq 0,\quad  i\in[N],\\
&\ \bM_{k-d_j}(g_j\by_i)\succeq 0,\quad i\in[N],j\in[m],\\
&\ L_{\by_1}(q_1)=1,\\
&\ L_{\by_{i}}(\bx^{\ba}q_i)=L_{\by_1}(\bx^{\ba}q_1),\quad \forall
\ba\in\N^n_{2k-\max\{\deg(q_1),\deg(q_i)\}},i\in[N]\setminus\{1\}.
\end{aligned}\right.\tag{P$k$}
\end{equation*}
It was shown in \cite{BHL2016} that under Assumption \ref{as::0} and the Archimedean condition, the sequence of optima of \eqref{eq::momentsdp} converges to $\rho$ as $k\to\infty$.

\section{The dual perspective and convergence rate analysis}\label{sec::dual}
In this section, we will unveil the underlying principle of the GMP reformulation \eqref{eq::gmp} for \eqref{eq::pro} from the dual perspective, which enables us to achieve a convergence rate analysis of the hierarchy of SDP relaxations for \eqref{eq::pro}.

\subsection{The dual perspective}
Let us derive the Lagrange dual problem of the GMP reformulation \eqref{eq::gmp} for \eqref{eq::pro}. Note that there are infinitely many
constraints involved in \eqref{eq::gmp}.
To formulate the dual problem, we need to embed these constraints into an appropriate functional space paired with a dual space \cite{shapiro2009semi}. Let
$\mathcal{Y}$ be the space of all functions $\omega:
\N^n\to\RR$ equipped with natural algebraic operations of addition and
multiplication by a scalar.
We associate this space with the dual space $\mathcal{Y}^*$ consisting of
functions $\omega^*: \N^n\to\RR$ such that only a finite number of
values $\omega^*(\ba), \ba\in\N^n$ are nonzero. For
$\omega\in\mathcal{Y}$ and $\omega^*\in\mathcal{Y}^*$, define the
scalar product 
\[\langle \omega^*, \omega\rangle\coloneqq\sum_{\ba\in\N^n}
\omega^*(\ba)\omega(\ba),
\]
where the summation is performed over $\ba$
in the finite support set of $\omega^*$. Equivalently, we can take
$\mathcal{Y}^*$ to be the polynomial ring $\RR[\bx]$. 
For any
$h(\bx)=\sum_{\ba\in\N^n}h_{\ba}\bx^{\ba}\in\RR[\bx]$ and
$\omega\in\mathcal{Y}$, define the scalar product 
\[\langle h, \omega\rangle\coloneqq\sum_{\ba\in\N^n}h_{\ba}\omega(\ba).
\]
Clearly, for each $i\in[N]\setminus\{1\}$, the mapping $\omega_i:\N^n\to \RR$
defined by 
\[
\omega_i(\ba)\coloneqq\int_{\K}
\bx^{\ba}q_i\ud\mu_i-\int_{\K}\bx^{\ba}q_1\ud\mu_1, \ \ \forall
\ba\in\N^n,
\]
belongs to the space $\mathcal{Y}$, and then the second part of constraints of
\eqref{eq::gmp} can be expressed as $\omega_i=0,i\in[N]\setminus\{1\}$. Let $h_i=\sum_{\ba\in\N^n}h^i_{\ba}\bx^{\ba}\in\RR[\bx]$ be the dual variable associated with the constraint $\omega_i=0,i\in[N]\setminus\{1\}$,
and let $c$ be the dual variable associated with the constraint $\int_{\K} q_1\ud\mu_1=1$.
The Lagrangian of \eqref{eq::gmp} now can be written as
\[\begin{aligned}
		\mathcal{L}(\{\mu_i\}, \{h_i\}, c)\coloneqq&\sum_{i=1}^N
		\int_{\K} p_i\ud
		\mu_i-c\left(\int_{\K}q_1\ud\mu_1-1\right)-\sum_{i=2}^N\langle h_i, 
            \omega_i\rangle\\
            =&\,c+\sum_{i=1}^N
		\int_{\K} p_i\ud
		\mu_i-\int_{\K}cq_1\ud\mu_1-\sum_{i=2}^N\sum_{\ba\in\N^n}
		h^i_{\ba}\left(\int_{\K}
		\bx^{\ba}q_i\ud\mu_i-\int_{\K}\bx^{\ba}q_1\ud\mu_1\right)\\
		=&\,c+\int_{\K}
		\left(p_1-cq_1+q_1\sum_{i=2}^Nh_i\right)\ud\mu_1
		+\sum_{i=2}^N \int_{\K} \left(p_i-q_i h_i\right)\ud\mu_i.
\end{aligned}\]
The Lagrange dual function of \eqref{eq::gmp} is 
\[\inf_{\mu_i\in\mathcal{M}(\K)_+}\mathcal{L}(\{\mu_i\}, \{h_i\}, c)=\left\{
\begin{array}{ll}
				c, & \text{if}\ p_1-q_1(c-\sum_{i=2}^Nh_i)
				\ge 0,\ p_i-q_i h_i\ge 0\ \text{on}\ \K,\\
				-\infty, &\text{otherwise}.
\end{array}\right.\]
Hence, the Lagrange dual problem of \eqref{eq::gmp} is 
\begin{equation*}\label{eq::gmpdual}
\left\{
	\begin{aligned}
		\sup_{c,h_i}&\ c\\
		\text{s.t.}&\ p_1(\bx)+\left(\sum_{i=2}^Nh_i(\bx)-c\right)q_1(\bx)\ge0,\quad\forall\bx\in\K,\\
		&\ p_i(\bx)-h_i(\bx)q_i(\bx)\ge0,\quad\forall \bx\in\K,i\in[N]\setminus\{1\},\\
  &\ h_i\in\RR[\bx],\quad i\in[N]\setminus\{1\}.
	\end{aligned}
	\right.\tag{D}
\end{equation*}
For each $i\in[N]\setminus\{1\}$, define 
\[\mathcal{H}_i\coloneqq\left\{h\in\RR[\bx]\,\middle|\,
		\frac{p_i(\bx)}{q_i(\bx)}\ge h(\bx),\ \ \forall
	\bx\in\K\right\}.\]
Since each $q_i>0$ on $\K$, the Lagrange dual problem of \eqref{eq::gmp} can be also expressed as
\[\sup_{h_i\in\mathcal{H}_i}\inf_{\bx\in\K}\
\frac{p_1(\bx)}{q_1(\bx)}+\sum_{i=2}^Nh_i(\bx).\]
Consequently, the underlying principle of the GMP
reformulation \eqref{eq::gmp} for \eqref{eq::pro}, derived from the dual aspect,
can be unfolded as follows: 1) replacing the terms
$\frac{p_i}{q_i}$, $i\in[N]\setminus\{1\}$ by polynomial approximations
$h_i\in\RR[\bx]$ from below on $\K$; 2) computing the infimum of
the resulting function which contains only a single denominator; 3) letting the approximations $h_i$'s vary and taking the supremum.

\begin{theorem}\label{th::dualgmpeq}
Under Assumption \ref{as::0}, the optimum of \eqref{eq::gmpdual} equals $\rho$.
\end{theorem}
\begin{proof}
Denote the optimum of \eqref{eq::gmpdual} by $\tau$.
Clearly, we have $\tau\le \rho$. Fix an arbitrary $\varepsilon>0$. By the Stone-Weierstrass theorem, for each $i\in[N]\setminus\{1\}$, there exists $\hat{h}_i\in\RR[\bx]$ such that
\[\sup_{\bx\in\K}\left|\frac{p_i(\bx)}{q_i(\bx)}-\hat{h}_i(\bx)\right|\le
\frac{\varepsilon}{2(N-1)}.\]
Letting $h_i\coloneqq \hat{h}_i-\varepsilon/2(N-1)$, we have 
\[\frac{p_i(\bx)}{q_i(\bx)}-\frac{\varepsilon}{N-1}\le h_i(\bx)\le \frac{p_i(\bx)}{q_i(\bx)},\quad\forall \bx\in\K.\]
Hence, $h_i\in\mathcal{H}_i$ and
\[\tau\ge \inf_{\bx\in\K}\frac{p_1(\bx)}{q_1(\bx)}+\sum_{i=2}^N h_i(\bx)\ge
\inf_{\bx\in\K}\sum_{i=1}^N\frac{p_i(\bx)}{q_i(\bx)}-\varepsilon\ge
\rho-\varepsilon.\]
Thus, $(\rho-\varepsilon, h_2,\ldots,h_N)$ is feasible to \eqref{eq::gmpdual}.
As $\varepsilon$ is arbitrary, we obtain $\tau=\rho$.
\end{proof}

\begin{remark}\label{rk::duality}
Theorem \ref{th::dualgmpeq}	can also be derived from strong duality between \eqref{eq::gmp} and \eqref{eq::gmpdual}. In fact, as $\K$ is compact, there exist
{\itshape constant} polynomials $\hat{h}_i$'s and $\hat{c}\in\RR$ such that $p_i/q_i>\hat{h}_i$, $i\in[N]\setminus\{1\}$, and
$p_1/q_1+\sum_{i=2}^N\hat{h}_i>\hat{c}$ on $\K$.
That is, $(\hat{c},\hat{h}_2,\ldots,\hat{h}_N)$ is a strictly feasible solution of \eqref{eq::gmpdual} and hence there is no dual gap between \eqref{eq::gmp} and \eqref{eq::gmpdual}.
\end{remark}

By truncating the polynomial degree, we obtain the following hierarchy of dual SOS relaxations for \eqref{eq::pro} ($k\ge d_{\min}$):
\begin{equation*}\label{eq::gmpdualsos}
\rho_k\coloneqq\left\{
	\begin{aligned}
		\sup_{c, h_i}&\ c\\
		\text{s.t.}&\
		p_1(\bx)+\left(\sum_{i=2}^Nh_i(\bx)-c\right)q_1(\bx)\in Q_k(\bg),\\
		&\ p_i(\bx)-h_i(\bx)q_i(\bx)\in Q_k(\bg),\quad i\in[N]\setminus\{1\},\\
            &\ h_i\in\RR[\bx]_{2k-\max\{\deg(q_1),\deg(q_i)\}},\quad i\in[N]\setminus\{1\}.
	\end{aligned}
	\right.\tag{D$k$}
\end{equation*}
\begin{theorem}
Under Assumption \ref{as::0} and the Archimedean condition, it holds $\rho_k\nearrow \rho$ as $k\to\infty$.
\end{theorem}
\begin{proof}
It is obvious that $\rho_k\le \rho_{k+1}\le \rho$ for all $k\ge d_{\min}$.
Fix an arbitrary $\varepsilon>0$ and we will prove that there exists $k_{\varepsilon}\in\N$ such that $\rho_{k_{\bar{\varepsilon}}}\ge \rho-\varepsilon$. For $\bar{\varepsilon}\coloneqq\varepsilon/3$, let
$(\rho-\bar{\varepsilon}, \bar{h}_2,\ldots,\bar{h}_N)$ be the feasible solution to \eqref{eq::gmpdual} provided in the proof of Theorem \ref{th::dualgmpeq}. For each $i=2,\ldots,N$,
let 
\[h_i\coloneqq \bar{h}_i-\frac{\bar{\varepsilon}}{N-1}.\]
Then, for $\bx\in\K$, it holds that
\begin{equation}\label{eq::positive}
p_i(\bx)-h_i(\bx)q_i(\bx)=p_i(\bx)-\bar{h}_i(\bx)q_i(\bx)+\frac{\bar{\varepsilon}}{N-1}q_i(\bx)\ge\frac{\bar{\varepsilon}}{N-1}q_i(\bx)>0.
\end{equation}
By Theorem \ref{th::PP}, there exists $k_i\in\N$ such that $h_i\in\RR[\bx]_{2k_i-\deg(q_i)}$ and $p_i-h_iq_i\in Q_{k_i}(\bg)$. Moreover, for $\bx\in\K$, we have 
\[\begin{aligned}
&p_1(\bx)+\left(\sum_{i=2}^N h_i(\bx)-(\rho-\varepsilon)\right)q_1(\bx)\\
=\,&p_1(\bx)+q_1(\bx)\sum_{i=2}^N \bar{h}_i(\bx)-\bar{\varepsilon} q_1(\bx)-\rho q_1(\bx)+3\bar{\varepsilon} q_1(\bx)\\
\ge\, & (\rho-\bar{\varepsilon})q_1(\bx)-\rho q_1(\bx)+2\bar{\varepsilon} q_1(\bx)
\ge \bar{\varepsilon} q_1(\bx)>0. 
\end{aligned}\]
By Theorem \ref{th::PP} again, there exists $k_1\in\N$
such that
\[p_1(\bx)+\left(\sum_{i=2}^N h_i(\bx)-(\rho-\varepsilon)\right)q_1(\bx)
\in Q_{k_1}(\bg).\]	
Let $k_{\varepsilon}\coloneqq\max_{1\le i\le N} k_i$. Then 
$(\rho-\varepsilon, h_2,\ldots,h_N)$ is feasible to \eqref{eq::gmpdualsos} with $k=k_{\varepsilon}$,
which implies $\rho_{k_{\varepsilon}}\ge \rho-\varepsilon$. 
\end{proof}

\subsection{Convergence rate analysis}
The exploration of the dual aspect of the GMP reformulation \eqref{eq::gmp} allows us to 
perform a convergence rate analysis for the hierarchy of SDP relaxations for \eqref{eq::pro} 
by utilizing some existing results from the literature. 
\begin{lemma}\label{prop::jackson}
For each $i\in[N]\setminus\{1\}$, there exists a constant $c_i$ depending on
$p_i$, $q_i$ and $\K$ such that for any $k\in\N$, there exists
$h\in\RR[\bx]_k$ satisfying 
\[
	\sup_{\bx\in\K} \left|\frac{p_i(\bx)}{q_i(\bx)}-h(\bx)\right|\le
	\frac{c_i}{k}.
\]
\end{lemma}
\begin{proof}
Denote by 
\[
	\omega_{\K}\left(\frac{p_i}{q_i},t\right)\coloneqq\sup\left\{\left|\frac{p_i(\bx)}{q_i(\bx)}
	-\frac{p_i(\by)}{q_i(\by)}\right| \colon \bx,\by\in\K, \
\Vert\bx-\by\Vert\le t\right\}
\]
the standard modulus of continuity of $p_i/q_i$ on $\K$.
By the multivariate version of Jackson's theorem (see \cite{Timan1963}), there exists
a constant $\hat{c}_i$ depending on $p_i$, $q_i$ and $\K$ such that for any
$k\in\N$, there exists $h\in\RR[\bx]_k$ satisfying 
\[
	\sup_{\bx\in\K} \left|\frac{p_i(\bx)}{q_i(\bx)}-h(\bx)\right|\le
	\hat{c}_i \omega_{\K}\left(\frac{p_i}{q_i},\frac{1}{k}\right).
\]
As $q_i(\bx)>0$ on $\K$, $p_i/q_i$ is Lipschitz on $\K$. So there is
a constant $L_i$ such that $\omega_{\K}(p_i/q_i, t)\le L_it$. Then, the
conclusion follows by letting $c_i\coloneqq\hat{c}_iL_i$.
\end{proof}

\begin{assumption}\label{assump::0}
The origin belongs to the interior of $\K$.
\end{assumption}

The following result is a consequence of the
fundamental result \cite[Theorem 6]{NieSchweighofer} and
\cite[Corollary 1]{KORDA2017}.
\begin{theorem}\cite[Theorem 3]{KH2018}\label{th::d}
	Let the Archimedean condition and Assumption~\ref{assump::0}
	hold. If $\phi(\bx)\in\RR[\bx]$ is strictly positive on $\K$, then
	$\phi\in Q_k(\bg)$ whenever
	\[
		k\ge C_{\K}
		\exp{\left[\left(3^{\deg(\phi)+1}\kappa^{\deg(\phi)}(\deg(\phi))^2n^{\deg{\phi}}
		\frac{\max_{\bx\in\K}\phi(\bx)}{\min_{\bx\in\K}\phi(\bx)}\right)^{C_{\K}}\right]},
	\]
	for some constant $C_{\K}$ depending only on $g_j$'s and 
	\begin{equation}\label{eq::r}
		\kappa\coloneqq\frac{1}{\sup\,\{t>0 \colon [-t,\ t]^n\subseteq\K\}}.
	\end{equation}
\end{theorem}

Let 
\[
	q^{\max}\coloneqq\max_{2\le i \le N, \bx\in\K}q_i(\bx),\ \
	q^{\min}\coloneqq\min_{2\le i\le N, \bx\in\K}q_i(\bx),\ \
	\rho^{\max}\coloneqq\max_{\bx\in\K}\sum_{i=1}^N\frac{p_i(\bx)}{q_i(\bx)}.
\]
Using Theorem \ref{th::d}, we can establish the following convergence rate of the hierarchy \eqref{eq::gmpdualsos}. 
\begin{theorem}
Let the Archimedean condition and Assumptions \ref{as::0}, 
\ref{assump::0} hold.
Then, there exist constants $C_1$ and $C_2$ depending on $p_i$'s,
$q_i$'s, $g_j$'s and $\K$, such that for any $\varepsilon>0$, we have
$\rho_k\ge \rho-\varepsilon$ whenever
\[
	\begin{aligned}
		k&\ge 
C_2
		\exp{\left[\left(3(3\omega n)^{D(N,\varepsilon)}D(N,\varepsilon)^2
		\frac{(4(\rho^{\max}-\rho)+3\varepsilon)(N-1)}{\varepsilon}\frac{q^{\max}}{q^{\min}}\right)^{C_2}\right]}\\
		&=O\left(\exp\left[\frac{1}{\varepsilon^{3C_2}}(3\omega
		n)^{\frac{4(N-1)C_1C_2}{\varepsilon}}\right]\right),
	\end{aligned}
\]
where $\omega=\max\{1, \kappa\}$ and 
\[
	D(N,\varepsilon)\coloneqq\max\left\{\deg(p_i),\ \left\lceil\frac{4C_1(N-1)}{\varepsilon}\right\rceil+\deg(q_i) :i\in[N]\right\}.
\]
\end{theorem}
\begin{proof}
	Fix an arbitrary $\varepsilon>0$. For each $i\in[N]\setminus\{1\}$, let
	$c_i$ be the constant in Lemma \ref{prop::jackson} and
	$C_1\coloneqq\max_{2\le i\le N}c_i$. Then,
	there exists $h_i\in\RR[\bx]$ of degree $\lceil
	4c_i(N-1)/\varepsilon\rceil$ satisfying 
\[
	\sup_{\bx\in\K} \left|\frac{p_i(\bx)}{q_i(\bx)}-h_i(\bx)\right|
	\le \frac{c_i}{\lceil 4c_i(N-1)/\varepsilon\rceil}
	\le \frac{\varepsilon}{4(N-1)}.
\]
For $i\in[N]\setminus\{1\}$, let $\hat{h}_i\coloneqq h_i-\varepsilon/2(N-1)$. 
Then, for any $\bx\in\K$, it holds that
\[
	\begin{aligned}
		p_i(\bx)-\hat{h}_i(\bx)q_i(\bx)=q_i(\bx)\left(\frac{p_i(\bx)}{q_i(\bx)}-\hat{h}_i(\bx)\right)
		=q_i(\bx)\left(\frac{p_i(\bx)}{q_i(\bx)}-h_i(\bx)+\frac{\varepsilon}{2(N-1)}\right),
	\end{aligned}
\]
and 
\[
		p_1(\bx)+\left(\sum_{i=2}^N\hat{h}_i(\bx)-(\rho-\varepsilon)\right)q_1(\bx)
		=q_1(\bx)\left(\frac{p_1(\bx)}{q_1(\bx)}+\sum_{i=2}^Nh_i(\bx)-\rho+\frac{\varepsilon}{2}\right).
\]
Therefore, for any $\bx\in\K$,  
\[
	\begin{aligned}
	p_i(\bx)-\hat{h}_i(\bx)q_i(\bx)&\ge \frac{\varepsilon}{4(N-1)}q^{\min},\\
		p_1(\bx)+\left(\sum_{i=2}^N\hat{h}_i(\bx)-(\rho-\varepsilon)\right)q_1(\bx)
		&\ge
		q_1(\bx)\left(\sum_{i=1}^N\frac{p_i(\bx)}{q_i(\bx)}-\frac{\varepsilon}{4}-\rho+\frac{\varepsilon}{2}\right)\ge \frac{\varepsilon}{4}q^{\min},
	\end{aligned}
\]
and 
\[
		\begin{aligned}
	p_i(\bx)-\hat{h}_i(\bx)q_i(\bx) & \le \frac{3\varepsilon}{4(N-1)} q^{\max},\\
		p_1(\bx)+\left(\sum_{i=2}^N\hat{h}_i(\bx)-(\rho-\varepsilon)\right)q_1(\bx)
		&\le
		q_1(\bx)\left(\sum_{i=1}^N\frac{p_i(\bx)}{q_i(\bx)}+\frac{\varepsilon}{4}-\rho+\frac{\varepsilon}{2}\right)\\
		&\le
		\left(\rho^{\max}-\rho+\frac{3\varepsilon}{4}\right) q^{\max}.
	\end{aligned}
\]
Note that 
\[
\begin{aligned}
&\frac{\varepsilon}{4(N-1)}q^{\min}\le \frac{\varepsilon}{4}q^{\min},\quad \frac{3\varepsilon}{4(N-1)} q^{\max}\le
\left(\rho^{\max}-\rho+\frac{3\varepsilon}{4}\right) q^{\max},\\
	&\deg(p_i-\hat{h}_i(\bx)q_i), \
	\deg\left(p_1+\left(\sum_{i=2}^N\hat{h}_i-(\rho-\varepsilon)\right)q_1\right)
	\le  D(N,\varepsilon).
\end{aligned}
\]
Hence, by Theorem \ref{th::d}, there exists a constant $C_2$ depending
on $g_j$'s such that whenever
\[
	\begin{aligned}
		k&\ge C_2
		\exp{\left[\left(3^{D(N,\varepsilon)+1}\omega^{D(N,\varepsilon)}D(N,\varepsilon)^2n^{D(N,\varepsilon)}
			\frac{\left(\rho^{\max}-\rho+\frac{3\varepsilon}{4}\right) q^{\max}}
		{\frac{\varepsilon}{4(N-1)}q^{\min}}\right)^{C_2}\right]}\\
&= 
C_2
		\exp{\left[\left(3(3\omega n)^{D(N,\varepsilon)}D(N,\varepsilon)^2
		\frac{(4(\rho^{\max}-\rho)+3\varepsilon)(N-1)}{\varepsilon}\frac{q^{\max}}{q^{\min}}\right)^{C_2}\right]}\\
		&=O\left(\exp\left[\frac{1}{\varepsilon^{3C_2}}(3\omega
		n)^{\frac{4(N-1)C_1C_2}{\varepsilon}}\right]\right),
	\end{aligned}
\]
$(\rho-\varepsilon, \hat{h}_2,\ldots,\hat{h}_N)$ is feasible to
\eqref{eq::gmpdualsos} and thus $\rho_k\ge \rho-\varepsilon$. 
\end{proof}

\section{Sign symmetry adapted SDP relaxations}\label{sec::sparsity}
In this section, we propose block-diagonal SDP relaxations for \eqref{eq::pro} by exploiting sign symmetries as well as correlative sparsity, which would significantly reduce the computational burden for structured problems.

\subsection{Sparse SDP relaxations by exploiting sign symmetries}
Let us first define the sign symmetries of a subset $\mathcal{A}\subseteq\N^n$.
\begin{definition}
Given a finite set $\mathcal{A}\subseteq\N^n$, the sign symmetries $\bR(\mathcal{A})$ of
$\mathcal{A}$ consist of all vectors $\br\in\Z^n_2\coloneqq\{0,1\}^n$ 
satisfying $\br^{\intercal}\ba\equiv0\,(\rm{mod}\,2)$ for all $\ba\in\mathcal{A}$. 
Moreover, we define the associated set $\overline{\bA}\coloneqq\{\ba\in\N^n\mid\br^{\intercal}\ba\equiv0\,(\rm{mod}\,2),\,\forall\br\in\bR(\bA)\}$.
\end{definition}

Note that $\bR(\mathcal{A})$ is a linear subspace of $\Z^n_2$ for any $\mathcal{A}\subseteq\N^n$.
The notion of sign symmetries stems from the invariance of polynomials under sign flips on variables \cite{Lofberg2009,WML2021}. For any $\br\in\Z^n_2$, we define the map $\theta_{\br}\colon\RR[\bx]\rightarrow\RR[\bx]$ by $\theta_{\br}(f)(x_1,\ldots,x_n)=f((-1)^{r_1}x_1,\ldots,(-1)^{r_n}x_n)$. A polynomial $f$ is said to have the sign symmetry $\br$ if $\theta_{\br}(f)=f$.
Note that a polynomial $f$ has the set of sign symmetries $\bR\subseteq\Z^n_2$ if and only if $\supp{f}\subseteq\{\ba\in\N^n\mid\br^{\intercal}\ba\equiv0\,(\rm{mod}\,2),\,\forall\br\in\bR\}$.


For $i\in[N]\setminus\{1\}$, let
\begin{equation}\label{eq::Ai}
	\bA_i\coloneqq\supp{p_i}\cup\supp{q_i}\cup\bigcup_{j=1}^m\supp{g_j},
\end{equation}
and let
\begin{equation}\label{eq::A1}
\bA_1\coloneqq\supp{p_1}\cup\supp{q_1}\cup\bigcup_{i=2}^{N}\mathcal{A}_i.
\end{equation}
For the remainder of this section, we denote the set of sign symmetries of $\bA_i$ by $\bR_i$. 

Given $\br\in\Z_2^n$, let $\br(\ca)\coloneqq((-1)^{r_1}a_1,\ldots,(-1)^{r_n}a_n)$ for $\ca\in\RR^n$ and for a set $S\subseteq\RR^n$, let $\br(S)\coloneqq\{\br(\ca)\in\RR^n\mid\ca\in S\}$. 
\begin{lemma}\label{sec4:lm2}
One has $\br(\K)=\K$, $\theta_{\br}(p_i)=p_i$, and $\theta_{\br}(q_i)=q_i$ 
for all $\br\in\mathcal{R}_i$, $i\in[N]$.
\end{lemma}
\begin{proof}
It is straightforward to verify from the definitions.
\end{proof}

In order to take the inherent sign symmetries of \eqref{eq::pro} into account, let us consider the following sign symmetry adapted version of \eqref{eq::gmpdual}:
\begin{equation*}\label{eq::gmpdualss}
\rho^{\rm{s}}\coloneqq\left\{
	\begin{aligned}
		\sup_{c,h_i}&\ c\\
		\text{s.t.}&\ \frac{p_1(\bx)}{q_1(\bx)}+\sum_{i=2}^Nh_i(\bx)\ge c,\quad\forall \bx\in\K,\\
		&\ \frac{p_i(\bx)}{q_i(\bx)}\ge h_i(\bx),\quad\forall \bx\in\K,\\
  &\ c\in\RR,\ h_i\in\RR[\overline{\bA_i}],\quad i\in[N]\setminus\{1\}.
	\end{aligned}\right.\tag{SD}
\end{equation*}

\begin{theorem}\label{cor::sdual}
Under Assumption \ref{as::0}, it holds $\rho^{\rm{s}}=\rho$. 
\end{theorem}
\begin{proof}
Clearly, we have $\rho^{\rm{s}}\le \rho$ by Theorem \ref{th::dualgmpeq}. To show the converse, fix an arbitrary $\varepsilon>0$ and let 
$(\rho-\varepsilon, h_2, \ldots, h_N)$ be the feasible solution to
\eqref{eq::gmpdual}  provided in the proof of Theorem \ref{th::dualgmpeq} so that
\begin{equation}\label{sec4:eq1}
\frac{p_i(\bx)}{q_i(\bx)}-\frac{\varepsilon}{N-1}\le h_i(\bx)\le \frac{p_i(\bx)}{q_i(\bx)},\quad\forall \bx\in\K.
\end{equation}
For each $i\in[N]\setminus\{1\}$, let $\tilde{h}_i\coloneqq\frac{1}{|\bR_i|}\sum_{\br\in\bR_i}\theta_{\br}(h_i)$. 
We have $\sum_{\br\in\bR_i}\theta_{\br}(\bx^{\ba})=0$ for each $\ba\in\N^n\setminus\overline{\bA_i}$.
In fact, as $\ba\in\N^n\setminus\overline{\bA_i}$, there exists $\tilde{\br}\in\mathcal{R}_i$ such that 
\[
-\sum_{\br\in\bR_i}\theta_{\br}(\bx^{\ba})=\tilde{\br}\left(\sum_{\br\in\bR_i}\theta_{\br}(\bx^{\ba})\right)
=\sum_{\br\in\bR_i}\theta_{\tilde{\br}+\br}(\bx^{\ba})=\sum_{\br\in\tilde{\br}+\bR_i}\theta_{\br}(\bx^{\ba})
=\sum_{\br\in\bR_i}\theta_{\br}(\bx^{\ba}).
\]
Therefore, we have $\tilde{h}_i\in\RR[\overline{\bA_i}]$. Moreover, it follows from Lemma \ref{sec4:lm2} and
\eqref{sec4:eq1} that
\begin{equation}\label{sec4:eq2}
\frac{p_i(\bx)}{q_i(\bx)}-\frac{\varepsilon}{N-1}\le \tilde{h}_i(\bx)\le \frac{p_i(\bx)}{q_i(\bx)},\quad\forall \bx\in\K.
\end{equation}
Thus we have
\[\inf_{\bx\in\K}\frac{p_1(\bx)}{q_1(\bx)}+\sum_{i=2}^N\tilde{h}_i(\bx)\ge
\inf_{\bx\in\K}\sum_{i=1}^N\frac{p_i(\bx)}{q_i(\bx)}-\varepsilon=\rho-\varepsilon.\]
Therefore, $(\rho-\varepsilon,\tilde{h}_2,\ldots,\tilde{h}_N)$ is feasible to
\eqref{eq::gmpdualss}. As $\varepsilon>0$ is arbitrary, we obtain
$\rho^{\rm{s}}\ge \rho$.
\end{proof}

Given $\br\in\Z_2^n$ with $\br(\K)=\K$, for a measure $\mu\in\mathcal{M}(\K)_+$,  
we define a new measure $\mu^{\br}$ by $\mu^{\br}(S)=\mu(\br(S))$ for any Borel set $S\subseteq\K$. A measure is said to be \emph{invariant} with respect to the sign symmetries $\bR$ if $\mu^{\br}=\mu$ for all $\br\in\bR$.

\begin{lemma}\label{sec4:lm1}
Let $\br\in\Z_2^n$, $\mu\in\mathcal{M}(\K)_+$, and $f\in\RR[\bx]$.
\begin{enumerate}
    \item For $\ba\in\N^n$, $\int_{\K}\bx^{\ba}\ud\mu^{\br}=(-1)^{\br^{\intercal}\ba}\int_{\K}\bx^{\ba}\ud\mu$.
    \item If $\br\in\bR(\supp{f})$, then $\int_{\K}f\ud\mu^{\br}=\int_{\K}f\ud\mu$.
    \item If $\br\in\bR(\supp{f})$, then for $\ba\in\N^n$, $\int_{\K}\bx^{\ba}f\ud\mu^{\br}=(-1)^{\br^{\intercal}\ba}\int_{\K}\bx^{\ba}f\ud\mu$.
\end{enumerate}
\end{lemma}
\begin{proof}
It is straightforward to verify from the definitions.
\end{proof}

Let $\mathcal{M}(\K)^{\bR_i}_+$ denote the set of finite positive Borel measures that are invariant with respect to the sign symmetries $\bR_i$. Then the dual of \eqref{eq::gmpdualss} reads as
\begin{equation*}\label{eq::gmpss}
	\left\{
	\begin{aligned}
		\inf_{\mu_i\in\mathcal{M}(\K)^{\bR_i}_+}&\ \sum_{i=1}^N \int_{\K}
		p_i\ud \mu_i\\
		\text{s.t.}\quad\,\,&\ \int_{\K} q_1\ud\mu_1=1,\\
		&\ \int_{\K}
		\bx^{\ba}q_i\ud\mu_i=\int_{\K}\bx^{\ba}q_1\ud\mu_1,\quad\forall
		\ba\in\overline{\bA_i},i\in[N]\setminus\{1\}.\\
	\end{aligned}\right.\tag{SP}
\end{equation*}

We now point out that \eqref{eq::gmpss} cannot be viewed as a consequence of \eqref{eq::gmp} equipped with invariant measures. Indeed, the latter would involve additional constraints
\[\int_{\K}\bx^{\ba}q_1\ud\mu_1=0,\quad\forall\ba\in\overline{\bA_1}\setminus\bigcap_{i=2}^N\overline{\bA_i},\]
since for any $\ba\in\overline{\bA_1}\setminus\bigcap_{i=2}^N\overline{\bA_i}$, there exists $j\in[N]\setminus\{1\}$ such that $\ba\notin\overline{\bA_j}$ which implies $\int_{\K}\bx^{\ba}q_1\ud\mu_1=\int_{\K}\bx^{\ba}q_j\ud\mu_j=0$ by Lemma \ref{sec4:lm1}. This (somewhat surprising) fact particularly highlights the significance of considering the SOS problem \eqref{eq::gmpdual} when one exploits sign symmetries for \eqref{eq::pro}.


The sign symmetry adapted reformulation \eqref{eq::gmpdualss} allows us to consider block-diagonal SDP relaxations for \eqref{eq::pro}. For $\mathcal{A}\subseteq\N^n$ and $k\in\N$, let us define a binary matrix $B^{\mathcal{A}}_k$ indexed by $\N^n_k$ through
\[[B_k^{\bA}]_{\ba,\bb}=\left\{
		\begin{array}{ll}
			1, & \text{if}\ \ba+\bb\in\overline{\bA},\\
		0, & \text{otherwise}.
		\end{array}\right.\]
It could be easily seen that $B^{\mathcal{A}}_k$ is block-diagonal up to appropriate row/column permutations \cite{WML2021}. We define the sparse quadratic module $Q_k(\bg,\bA)$ associated with $\bA$ by 
\[Q_k(\bg,\bA)\coloneqq\left\{\sigma_0+\sum_{j=1}^m \sigma_j g_j \middle| 
\sigma_j=[\bx]_{k-d_j}^{\intercal}G_j[\bx]_{k-d_j},G_j\in\mathbb{S}^+(B^{\bA}_{k-d_j}),j\in\{0\}\cup[m]\right\},\]
where $d_0:=0$, for $s\in\N$, $[\bx]_s\coloneqq[1,x_1,x_2,\ldots,x_n^s]$ denotes the canonical vector of monomials up to degree $s$, and $\mathbb{S}^+(B^{\bA}_{s})$ denotes the set of positive semidefinite matrices with sparsity pattern being specified by $B^{\bA}_{s}$. 
The sign symmetry adapted version of \eqref{eq::gmpdualsos} is given by
\begin{equation*}\label{eq::gmpdualtssos}
\rho^{\rm{s}}_k\coloneqq\left\{
	\begin{aligned}
		\sup_{c,h_i}&\ c\\
		\text{s.t.}&\
		p_1(\bx)+\left(\sum_{i=2}^Nh_i(\bx)-c\right)q_1(\bx)\in
		Q_k(\bg,\bA_1),\\
		&\ p_i(\bx)-h_i(\bx)q_i(\bx)\in Q_k(\bg,\bA_i),\quad
		i\in[N]\setminus\{1\},\\
            &\ c\in\RR,\ h_i(\bx)\in\RR[\overline{\bA_i}]\cap\RR[\bx]_{2k-\max\{\deg(q_1),\deg(q_i)\}},\quad i\in[N]\setminus\{1\}.
	\end{aligned}
	\right.\tag{SD$k$}
\end{equation*}

The following theorem is a sign symmetry adapted version of Theorem \ref{th::PP}.
\begin{theorem}\cite[Theorem 6.11]{WML2021}\label{th::tsparse}
Let $f\in\RR[\bx]$ and $\mathcal{A}=\supp{f}\cup\bigcup_{j=1}^m\supp{g_j}$. Assume that the quadratic module $Q(\bg)$ is Archimedean and $f$ is positive on $\K$. Then $f\in
Q_k(\bg,\mathcal{A})$ for some $k\in\N$. 
\end{theorem}

\begin{remark}\label{rm}
By a similar argument as for Theorem \ref{th::tsparse} (see \cite{WML2021}), one can actually show that if $f\in Q_k(\bg)$, then $f\in Q_k(\bg,\mathcal{A})$ with $\mathcal{A}=\supp{f}\cup\bigcup_{j=1}^m\supp{g_j}$. 
\end{remark}

\begin{theorem}\label{th::convts}
The following statements hold true:
\begin{enumerate}
    \item[\upshape (i)] For $k\ge d_{\min}$, $\rho^{\rm{s}}_k\le \rho_k$. Moreover, if $\bR_1=\bR_2=\cdots=\bR_N$, then $\rho^{\rm{s}}_k=\rho_k$.
    \item[\upshape (ii)] Under Assumption \ref{as::0} and the Archimedean condition, $\rho^{\rm{s}}_k\nearrow \rho$ as $k\to\infty$.
\end{enumerate}
\end{theorem}
\begin{proof}
(i). It follows from the fact that any feasible solution $(c,h_2,\ldots,h_N)$ to \eqref{eq::gmpdualtssos} is also feasible to \eqref{eq::gmpdualsos}. Now assume $\bR_1=\bR_2=\cdots=\bR_N$. To show $\rho^{\rm{s}}_k\ge \rho_k$, let $(c, h_2, \ldots, h_N)$ be any feasible solution to \eqref{eq::gmpdualsos}.
For each $i\in[N]\setminus\{1\}$, let $\tilde{h}_i\coloneqq\frac{1}{|\bR_1|}\sum_{\br\in\bR_1}\theta_{\br}(h_i)\in\RR[\overline{\bA_1}]\cap\RR[\bx]_{2k-\max\{\deg(q_1),\deg(q_i)\}}$. From $p_i(\bx)-h_i(\bx)q_i(\bx)\in Q_k(\bg)$ and Remark \ref{rm}, we deduce that
\begin{equation}\label{sec4:eq3}
p_i(\bx)-\tilde{h}_i(\bx)q_i(\bx)\in Q_k(\bg,\bA_1).
\end{equation}
Moreover, from $p_1(\bx)+\left(\sum_{i=2}^Nh_i(\bx)-c\right)q_1(\bx)\in Q_k(\bg)$ and Remark \ref{rm} it follows that
\[p_1(\bx)+\left(\sum_{i=2}^N\tilde{h}_i(\bx)-c\right)q_1(\bx)\in Q_k(\bg,\bA_1).\]
Therefore, $(c,\tilde{h}_2,\ldots,\tilde{h}_N)$ is feasible to
\eqref{eq::gmpdualtssos}, which implies $\rho^{\rm{s}}_k\ge \rho_k$. 

(ii). It is clear that $\rho^{\rm{s}}_k\le \rho^{\rm{s}}_{k+1}$ for any $k\ge d_{\min}$.
To show the convergence, fix an arbitrary $\varepsilon>0$.
For $\bar{\varepsilon}\coloneqq\varepsilon/3$, let
$(\rho-\bar{\varepsilon}, \bar{h}_2,\ldots,\bar{h}_N)$ be the feasible solution to \eqref{eq::gmpdualss} provided in the proof of Theorem \ref{cor::sdual}. For each $i\in[N]\setminus\{1\}$,
let 
\[h_i\coloneqq \bar{h}_i-\frac{\bar{\varepsilon}}{N-1}\in\RR[\overline{\bA_i}].\]
Then, for $\bx\in\K$, it holds that
\begin{equation}\label{eq::positive1}
p_i(\bx)-h_i(\bx)q_i(\bx)=p_i(\bx)-\bar{h}_i(\bx)q_i(\bx)+\frac{\bar{\varepsilon}}{N-1}q_i(\bx)\ge\frac{\bar{\varepsilon}}{N-1}q_i(\bx)>0.
\end{equation}
By Theorem \ref{th::tsparse}, there exists $k_i\in\N$ such that $h_i\in\RR[\overline{\bA_i}]\cap\RR[\bx]_{2k_i-\deg(q_i)}$ and $p_i-h_iq_i\in Q_{k_i}(\bg,\bA_i)$. Moreover, for $\bx\in\K$, we have 
\[\begin{aligned}
&p_1(\bx)+\left(\sum_{i=2}^N h_i(\bx)-(\rho-\varepsilon)\right)q_1(\bx)\\
=\,&p_1(\bx)+q_1(\bx)\sum_{i=2}^N \bar{h}_i(\bx)-\bar{\varepsilon} q_1(\bx)-\rho q_1(\bx)+3\bar{\varepsilon} q_1(\bx)\\
\ge\, & (\rho-\bar{\varepsilon})q_1(\bx)-\rho q_1(\bx)+2\bar{\varepsilon} q_1(\bx)
\ge \bar{\varepsilon} q_1(\bx)>0. 
\end{aligned}\]
By Theorem \ref{th::tsparse} again, there exists $k_1\in\N$
such that
\[p_1(\bx)+\left(\sum_{i=2}^N h_i(\bx)-(\rho-\varepsilon)\right)q_1(\bx)
\in Q_{k_1}(\bg,\bA_1).\]	
Let $k_{\varepsilon}\coloneqq\max_{1\le i\le N} k_i$. Then 
$(\rho-\varepsilon, h_2,\ldots,h_N)$ is feasible to \eqref{eq::gmpdualtssos} with $k=k_{\varepsilon}$,
which implies $\rho^{\rm{s}}_{k_{\varepsilon}}\ge \rho-\varepsilon$. As $\varepsilon$ is arbitrary, we prove the convergence of $\{\rho^{\rm{s}}_k\}$ to $\rho$.
\end{proof}

Note that the optimum $\rho^{\rm{s}}_k$ of \eqref{eq::gmpdualtssos} may depend on which ratio being chosen as $p_1/q_1$. We give an example to illustrate this phenomenon.
\begin{example}\label{ex::0}
Let
\begin{equation*}
    f\coloneqq\frac{x^2+y^2-yz}{1+2x^2+y^2+z^2}+\frac{y^2+x^2z}{1+x^2+2y^2+z^2}+\frac{z^2-x+y}{1+x^2+y^2+2z^2}.
\end{equation*}
Consider the minimization of $f$ over the unit ball $\{(x,y,z)\in\RR^3\mid 1-x^2-y^2-z^2\ge0\}$. 
We consider three cases: (1) $p_1=x^2+y^2-yz,q_1=1+2x^2+y^2+z^2$; (2) $p_1=y^2+x^2z,q_1=1+x^2+2y^2+z^2$; 
(3) $p_1=z^2-x+y,q_1=1+x^2+y^2+2z^2$. We present the computational results in Table \ref{tab::0}.
For this problem, $-0.3465$ can be certified to be globally optimal.

\begin{table}[htbp]\label{tab::0}
\caption{Computational results for Example \ref{ex::0}.}
\centering
\begin{tabular}{cccccccc}
\midrule[0.8pt]
\multirow{2}{*}{$k$} &&\multirow{2}{*}{$\sup$\eqref{eq::gmpdualsos}} &&\multicolumn{3}{c}{$\sup$\eqref{eq::gmpdualtssos}} \\
 \cline{5-7}
&& & & case (1) & case (2) & case (3)\\
\midrule[0.4pt]
        $2$&&-0.3563 &&-0.4275 & -0.4513 & -0.4738 \\
        $3$&&-0.3465 &&-0.3469 & -0.3546 & -0.3550 \\
        $4$&&&&-0.3465 & -0.3465 & -0.3465\\
\midrule[0.8pt]
\end{tabular}
\end{table}

\end{example}

\begin{remark}
To guarantee $\rho^{\rm{s}}_k=\rho_k$, one may simply letting $\bA_1=\cdots=\bA_N$ by Theorem \ref{th::convts}.
\end{remark}

Finally, the dual of \eqref{eq::gmpdualtssos} reads as
\begin{equation*}\label{eq::momentsdpts}
\left\{
	\begin{aligned}
		\inf_{\by_i}&\ \sum_{i=1}^N L_{\by_{i}}(p_i)\\
		\text{s.t.}&\ B^{\bA_i}_{k}\circ \bM_k(\by_i)\succeq 0,\quad  i\in[N],\\
		&\ B^{\bA_i}_{k-d_j}\circ \bM_{k-d_j}(g_j\by_i)\succeq 0,\quad i\in[N],j\in[m],\\
		&\ L_{\by_1}(q_1)=1,\\
		&\ L_{\by_{i}}(\bx^{\ba}q_i)=L_{\by_1}(\bx^{\ba}q_1),\quad\forall
		\ba\in\N^n_{2k-\max\{\deg(q_1),\deg(q_i)\}}\cap\overline{\bA_i},i\in[N]\setminus\{1\}.\\
	\end{aligned}\right.\tag{SP$k$}
\end{equation*}

\subsection{Sparse SDP relaxations by exploiting both correlative sparsity and sign symmetries}
In this subsection, we present sparse SDP relaxations for \eqref{eq::pro} by exploiting both correlative sparsity and sign symmetries. Let us begin by recalling a correlative sparsity adapted GMP reformulation and a corresponding hierarchy of sparse SDP relaxations for \eqref{eq::pro} proposed in \cite{BHL2016}. We first describe the correlative sparsity pattern in \eqref{eq::pro}.
For $I\subseteq[n]$, $\bx(I)$ denotes the set of variables $x_i$ with $i\in I$.
\begin{assumption}\label{assump2}
The index sets $[n]$ and $[m]$ can be decomposed as $[n]=\cup_{i=1}^N I_i$ and $[m]=\cup_{i=1}^N J_i$ such that
\begin{enumerate}
       \item[(i)] For every $i\in[N]$, $p_i, q_i\in\RR[\bx(I_i)]$;
       \item[(ii)] For every $j\in J_i$, $g_j\in\RR[\bx(I_i)]$;
       \item[(iii)] For every $i\in[N]$, there exists $k\in J_i$ such that 
       $g_k=M_i-\sum_{\ell\in I_i}x_{\ell}^2$ for some $M_i>0$;
       \item[(iv)] The subsets $\{I_i\}_{i=1}^N$ satisfy the running intersection property (RIP), that is, 
       for every $i\in[N]\setminus\{1\}$, $I_i\cap (\cup_{j=1}^{i-1}I_j)\subseteq I_k$ for some $k\in[i-1]$. 
   \end{enumerate}
\end{assumption}

\begin{remark}
If $\K$ is compact and one knows some $M>0$ such that $M-x_1^2-\cdots-x_n^2\ge0$ for all $\bx\in\K$, then we can always add redundant constraints $M-\sum_{j\in I_i}x_j^2\ge0, i\in[N]$ to the description of $\K$ so that Assumption \ref{assump2} (iii) is satisfied.
\end{remark}

For every $i\in[N]$, let 
\[
\K_i\coloneqq\left\{\bx(I_i)\in\RR^{|I_i|} \mid g_j(\bx(I_i))\ge 0, \ j\in J_i\right\},
\]
and $\pi_i\colon \mathcal{M}(\K)_+\rightarrow\mathcal{M}(\K_i)_+$ be the projection on $\K_i$, that is, for any $\mu\in\mathcal{M}(\K)_+$, 
\begin{equation*}
    \pi_i(\mu(B))\coloneqq\mu(\{\bx:\bx\in \K,\bx(I_i)\in B\}), \quad\forall B\in\mathcal{B}(\K_i),
\end{equation*}
where $\mathcal{B}(\K_i)$ is the usual Borel $\sigma$-algebra associated with $\K_i$. 
For every pair $(i,j)\in[N]\times[N]$ with $i\ne j$ and $I_i\cap I_j\ne\emptyset$, let
\begin{equation*}
    \K_{ij}=\K_{ji}\coloneqq\left\{\bx(I_i\cap I_j):\bx(I_i)\in\K_i,\ \bx(I_j)\in\K_j\right\},
\end{equation*}
and the projection $\pi_{ij}\colon \mathcal{M}(\K_i)_+\rightarrow\mathcal{M}(\K_{ij})_+$ is defined in an obvious similar manner. 
Moreover, let us define the sets
\[U_i\coloneqq\{j\in\{i+1, \ldots, N\} \mid I_i\cap I_j\neq \emptyset\},\ i\in[N-1],\quad 
U_N\coloneqq\emptyset,\]
and 
\[V_1\coloneqq\emptyset,\quad V_i\coloneqq\{j\in\{1, \ldots, i-1\} \mid I_i\cap I_j\neq \emptyset\},\ i\in[N]\setminus\{1\}.\]
Then the correlative sparsity adapted GMP reformulation for \eqref{eq::pro} is given by
\begin{equation*}\label{eq::gmpcs}
\rho^{\rm{c}}\coloneqq\left\{
	\begin{aligned}
		\inf_{\mu_i\in\mathcal{M}(\K_i)_+}&\ \sum_{i=1}^N \int_{\K_i}
		p_i\ud \mu_i\\
		\text{s.t.}\quad\,\,&\ \int_{\K_i} q_i\ud\mu_i=1,\quad i\in[N],\\
		&\ \pi_{ij}(q_i\ud\mu_i)=\pi_{ji}(q_j\ud\mu_j),\quad  j\in U_i,\ i\in[N-1].
	\end{aligned}\right.\tag{CP}
\end{equation*}

For $I\subseteq [n]$ and $f\in\mathbb{R}[\bx(I)]$, let $\bM_k(\by, I)$ (resp. $\bM_k(f\by, I)$) be the moment (resp. localizing) submatrix obtained by retaining only those rows and columns of $\bM_k(\by)$ (resp. $\bM_k(f\by)$) indexed by $\ba\in \mathbb{N}^n$ with $\alpha_i=0$ whenever $i\notin I$.
For $1\le i<j\le N$ with $I_i\cap I_j\neq\emptyset$, 
let $\N^{(ij)}\coloneqq\{\ba\in\N^n \mid \alpha_k=0, \,\forall k\not\in I_i\cap I_j\}$.
Then, the hierarchy of correlative sparsity adapted SDP relaxations for \eqref{eq::pro} is given by
\begin{equation*}\label{eq::gmpcssdp}
\rho_k^{\rm{c}}\coloneqq\left\{
	\begin{aligned}
		\inf_{\by_i}&\ \sum_{i=1}^N L_{\by_{i}}(p_i)\\
		\text{s.t.}&\ \bM_k(\by_i, I_i)\succeq 0,\quad  i\in[N],\\
		&\ \bM_{k-d_j}(g_j\by_i, I_i)\succeq 0,\quad j\in J_i,i\in[N],\\
		&\ L_{\by_i}(q_i)=1,\quad i\in[N],\\
		&\ L_{\by_{i}}(\bx^{\ba}q_i)=L_{\by_j}(\bx^{\ba}q_j),\\
  &\ \forall
		\ba\in\N_{2k-\max\{\deg(q_i),\deg(q_j)\}}^{(ij)}, j\in U_i,i\in[N-1].\\
	\end{aligned}\right.\tag{CP$k$}
\end{equation*}

\begin{theorem}{\upshape \cite[Theorems 3.1 and 3.2]{BHL2016}}\label{th::eqcs}
Let Assumption \ref{assump2} hold and assume that $q_i>0$ on $\K_i$ for $i\in[N]$. 
It hold $\rho^{\rm{c}}=\rho$ and $\rho^{\rm{c}}_{k}\nearrow \rho$ as $k\to\infty$.
\end{theorem}

We can derive the Lagrange dual of \eqref{eq::gmpcssdp} which reads as
\begin{equation*}\label{eq::csssdpdual}
\left\{
\begin{aligned}
   \sup_{c_i, h_{i,j}}&\ \sum_{i=1}^N c_i\\
   \text{s.t.}\,\,&\ p_i-\left(c_i+\sum_{j\in U_i}h_{i,j}-\sum_{j\in V_i}h_{j,i}\right)q_i\in
   Q_k(\{g_j\}_{j\in J_i}),\ c_i\in\RR, \ i\in[N],\\
   &\ h_{i,j}\in\RR[\bx(I_i\cap I_j)]_{2k-\max\{\deg(q_i),\deg(q_j)\}},\quad 
    j\in U_i,\ i\in[N-1].
\end{aligned}
\right.\tag{CD$k$}
\end{equation*} 
\begin{proposition}\label{prop::sd}
    Under Assumption \ref{assump2} (iii), the strong duality holds between \eqref{eq::gmpcssdp}
    and \eqref{eq::csssdpdual} for all $k\ge d_{\min}$.
\end{proposition}
\begin{proof}
Let $\{\mu_i\}$ be a feasible solution of \eqref{eq::gmpcs} (see the proof of \cite[Theorem 3.1]{BHL2016}
for the existence). For each $i\in[N]$, let $\by_i=(y_{i\ba})_{\ba\in\N^{|I_i|}_{2k}}$ be such that 
$y_{i\ba}=\int_{\K_i}\bx^{\ba}\ud\mu_i$ for $\ba\in\N^{|I_i|}_{2k}$.
Then, $\{\by_i\}$ is a feasible solution of \eqref{eq::gmpcssdp}. 
As Assumption \ref{assump2} (iii) holds, 
according to the proof of \cite[Theorem 2.2]{BHL2016},
the feasible set of \eqref{eq::gmpcssdp} is compact, which implies that
the optimal solution set of \eqref{eq::gmpcssdp} is nonempty and bounded.
Therefore, \eqref{eq::csssdpdual} is strictly feasible \cite{Tr2005}
and the strong duality holds \cite[Theorem 4.1.3]{SS2000}. 
\end{proof}

Let
\begin{equation}
\bA\coloneqq\bigcup_{i=1}^{N}\left(\supp{p_i}\cup\supp{q_i}\right)\cup\bigcup_{j=1}^{m}\supp{g_j},
\end{equation}
and $\bR$ be the set of sign symmetries of $\bA$.
Now we consider the following sign symmetry adapted version of \eqref{eq::gmpcs}:
\begin{equation*}\label{eq::gmpcss}
\rho^{\rm{cs}}\coloneqq\left\{
	\begin{aligned}
		\inf_{\mu_i\in\mathcal{M}(\K_i)_+^{\bR}}&\ \sum_{i=1}^N \int_{\K_i}
		p_i\ud \mu_i\\
		\text{s.t.}\quad\,\,&\ \int_{\K_i} q_i\ud\mu_i=1,\quad i\in[N],\\
		&\ \pi_{ij}(q_i\ud\mu_i)=\pi_{ji}(q_j\ud\mu_j),\quad  j\in U_i, i\in[N-1].
	\end{aligned}\right.\tag{CSP}
\end{equation*}

\begin{theorem}
Let Assumption \ref{assump2} hold and assume that $q_i>0$ on $\K_i$ for $i\in[N]$. 
It holds $\rho^{\rm{cs}}=\rho$.
\end{theorem}
\begin{proof}
Since $\mathcal{M}(\K_i)^{\bR}_+$ is a subset of $\mathcal{M}(\K_i)_+$ for each $i\in[N]$, we immediately have $\rho^{\rm{cs}}\ge\rho$. For the converse, fix an arbitrary $\varepsilon>0$ and 
suppose that $\{\mu_i\}$ is any feasible solution of \eqref{eq::gmpcs} with 
$\sum_{i=1}^N \int_{\K}p_i\ud\mu_i<\rho+\varepsilon$. 
For each $i\in[N]$, let us define a new measure $\tilde{\mu}_i\in\mathcal{M}(\K_i)^{\bR}_+$ by $\tilde{\mu}_i\coloneqq\frac{1}{|\bR|}\sum_{\br\in\bR}\mu_i^{\br}$. 
For any $j\in U_i, i\in[N-1]$, by Lemma \ref{sec4:lm1}, one can check that 
$\pi_{ij}(q_i\ud\tilde{\mu_i})=\pi_{ji}(q_j\ud\tilde{\mu}_j).$
Moreover, by Lemma \ref{sec4:lm1} again, $\sum_{i=1}^N \int_{\K}p_i\ud\tilde{\mu}_i=\sum_{i=1}^N \int_{\K}p_i\ud\mu_i<\rho+\varepsilon$. 
As $\varepsilon>0$ is arbitrary, $\rho^{\rm{cs}}\le\rho$. 
\end{proof}

For each $i\in[N]$ and each $k\in\N$, define a binary matrix $B^{\bA}_{i,k}$ indexed by $\N^{|I_i|}_k$ (we embed $\N^{|I_i|}_k$ into $\N^{n}_k$ in the natural way) such that
\[[B^{\bA}_{i,k}]_{\ba,\bb}\coloneqq\left\{
		\begin{array}{ll}
			1, & \text{if}\ \ba+\bb\in
			\overline{\bA},\\
		0, & \text{otherwise}.
		\end{array}\right.\]
The sign symmetry adapted version of \eqref{eq::gmpcssdp} is given by
\begin{equation*}\label{eq::gmpcsssdp}
\rho^{\rm{cs}}_{k}\coloneqq\left\{
	\begin{aligned}
		\inf_{\by_i}&\ \sum_{i=1}^N L_{\by_{i}}(p_i)\\
		\text{s.t.}&\ B^{\bA}_{i,k}\circ \bM_k(\by_i, I_i)\succeq 0,\quad i\in[N],\\
		&\ B^{\bA}_{i,k-d_j}\circ \bM_{k-d_j}(g_j\by_i, I_i)\succeq 0,\quad j\in J_i,i\in[N],\\
		&\ L_{\by_i}(q_i)=1,\quad i\in[N],\\
		&\ L_{\by_{i}}(\bx^{\ba}q_i)=L_{\by_j}(\bx^{\ba}q_j),\\
  &\ \forall\ba\in\N_{2k-\max\{\deg(q_i),\deg(q_j)\}}^{(ij)}\cap\overline{\bA},j\in U_i,i\in[N-1].
	\end{aligned}\right.\tag{CSP$k$}
\end{equation*}

\begin{theorem}\label{th::eqcss}
Let Assumption \ref{assump2} hold and assume that $q_i>0$ on $\K_i$ for $i\in[N]$. 
It holds $\rho^{\rm{cs}}_{k}=\rho^{\rm{c}}_{k}$ for all $k\ge d_{\min}$. Consequently, $\rho^{\rm{cs}}_{k}\nearrow \rho$ as $k\to\infty$.
\end{theorem}
\begin{proof}
Let $\{\by_i\}$ be a feasible solution to \eqref{eq::gmpcssdp}.
Note that $B^{\mathcal{A}}_{i,k}\circ \bM_k(\by_i, I_i)$ (resp. $B^{\mathcal{A}}_{i,k-d_j}\circ
\bM_{k-d_j}(g_j\by_i, I_i)$, $j\in J_i$) consists of diagonal blocks of $\bM_k(\by_i, I_i)$
(resp. $\bM_{k-d_j}(g_j\by_i, I_i)$, $j\in J_i$) for all $i\in[N]$. 
Thus, $\{\by_i\}$ is also feasible to \eqref{eq::gmpcsssdp}.
So $\rho^{\rm{cs}}_{k}\le\rho^{\rm{c}}_{k}$.

On the other hand, let $\{\by_i\}$ be any feasible solution to \eqref{eq::gmpcsssdp}. 
For every $i\in[N]$, we define a pseudo-moment sequence $\by'_i=(y'_{i{\ba}})_{\ba\in\N^{|I_i|}_k}$ 
as follows:
\begin{equation*}
    y'_{i{\ba}}=\begin{cases}
    y_{i\ba},&\text{ if }\ba\in\overline{\bA},\\
    0,&\text{ otherwise}.
    \end{cases}
\end{equation*}
By the definition of $\mathcal{A}$, one can easily check that $\{\by'_i\}$ is a feasible solution to \eqref{eq::gmpcssdp} and $\sum_{i=1}^N L_{\by'_i}(p_i)=\sum_{i=1}^N L_{\by_{i}}(p_i)$. So $\rho^{\rm{cs}}_{k}\ge\rho^{\rm{c}}_{k}$ and it follows $\rho^{\rm{cs}}_{k}=\rho^{\rm{c}}_{k}$ as desired.
\end{proof}

The dual of \eqref{eq::gmpcsssdp} reads as
\begin{equation*}\label{eq::csssdp}
\left\{
\begin{aligned}
   \sup_{c_i, h_{i,j}}&\ \sum_{i=1}^N c_i\\
   \text{s.t.}\,\,&\ p_i-\left(c_i+\sum_{j\in U_i}h_{i,j}-\sum_{j\in V_i}h_{j,i}\right)q_i\in
   Q_k(\{g_j\}_{j\in J_i},\bA),\ c_i\in\RR, \ i\in[N],\\
   &\ h_{i,j}\in\RR[\bx(I_i\cap I_j)]_{2k-\max\{\deg(q_i),\deg(q_j)\}}\cap\RR[\overline{\bA}],\quad 
    j\in U_i,i\in[N-1].
\end{aligned}
\right.\tag{CSD$k$}
\end{equation*} 
\begin{proposition}
    Under Assumption \ref{assump2} (iii), the strong duality holds between \eqref{eq::gmpcsssdp}
    and \eqref{eq::csssdp} for all $k\ge d_{\min}$.
\end{proposition}
\begin{proof}
   Note that for any feasible solution $\{\by_i\}$ of \eqref{eq::gmpcsssdp}, only the entries 
   $y_{i\ba}, \ba\in\overline{\mathcal{A}}$, are involved in \eqref{eq::gmpcsssdp}. By the proof of
   Theorem \ref{th::eqcss}, the corresponding point $\{\by'_i\}$ is feasible to \eqref{eq::gmpcssdp}.
   Then, the proof of Proposition \ref{prop::sd} indicates that the feasible set of \eqref{eq::gmpcsssdp}
   is compact. Hence, the strong duality holds between \eqref{eq::gmpcsssdp} and \eqref{eq::csssdp}
   as proved in Proposition \ref{prop::sd}.
\end{proof}

\section{Numerical experiments}\label{sec::num}
In this section, we conduct numerical experiments to test the performance of 
the sign symmetry adapted approaches against the approaches in \cite{BHL2016}.
We use the Julia package {\tt TSSOS}\footnote{{\tt TSSOS} is publically available at https://github.com/wangjie212/TSSOS.} 
\cite{magron2021tssos} to build the SDP relaxations and rely on {\tt MOSEK} \cite{mosek}
to solve them. Throughout this section, $k$ stands for the relaxation order.
All numerical experiments were performed on a desktop computer with Intel(R) Core(TM) i9-10900 CPU@2.80GHz and 64G RAM. 

\subsection{Comparison of \eqref{eq::gmpdualtssos} with \eqref{eq::gmpdualsos}}
We begin by presenting three examples without correlative sparsity.

\begin{example}\label{ex1}
Consider the problem
\begin{equation}\label{ex::pqa}
\min_{\bx\in\RR^3}\ \sum_{a=1/M}^{1-1/M}\frac{p_a(\bx)}{q_a(\bx)}\quad \text{s.t.} \ x_1^2+x_2^2+x_3^2=3,
\end{equation}
where
\[
\begin{aligned}
p_a(\bx)=&\,a^4(x_1^{6d}+x_2^{6d}+x_3^{6d})+(x_1^{4d}x_2^{2d}+x_2^{4d}x_3^{2d}+
x_3^{4d}x_1^{2d})+\\
&\,a^8(x_1^{2d}x_2^{4d}+x_2^{2d}x_3^{4d}+x_3^{2d}x_1^{4d}),\\
q_a(\bx)=&\,2a^6(x_1^{4d}x_2^{2d}+x_2^{4d}x_3^{2d}+x_3^{4d}x_1^{2d})+
    2a^2(x_1^{2d}x_2^{4d}+x_2^{2d}x_3^{4d}+x_3^{2d}x_1^{4d})+\\
    &\,3(1-2a^2+a^4-2a^6+a^8)x_1^{2d}x_2^{2d}x_3^{2d},
\end{aligned}
\]
and $M, d$ are positive integers. 
Replacing $x_i^d$'s by $y_i$'s in $p_a(\bx), q_a(\bx)$ and denoting the resulting
polynomials by $\tilde{p}_a(\by), \tilde{q}_a(\by)$ with $\by=(y_1, y_2, y_3)$, 
one can show that for all $0<a<1$, $\tilde{p}_a(\by)-\tilde{q}_a(\by)\ge 0$ on $\RR^3$
with ten zeros: $\{(1, \pm 1, \pm 1), (\pm a, 1, 0), (0, \pm a, 1), (1, 0, \pm a)\}$ \cite{Reznick96}.
Then, it is obvious that the optimum of \eqref{ex::pqa} is $M-1$.
Moreover, we can rewrite $q_a(\bx)$ as 
\[
\begin{aligned}
q_a(\bx)=&\,2a^6(x_1^{4d}x_2^{2d}+x_2^{4d}x_3^{2d}+x_3^{4d}x_1^{2d}-3x_1^{2d}x_2^{2d}x_3^{2d})+\\
&\,2a^2(x_1^{2d}x_2^{4d}+x_2^{2d}x_3^{4d}+x_3^{2d}x_1^{4d}-3x_1^{2d}x_2^{2d}x_3^{2d})+\\
    &\,3(1+a^4+a^8)x_1^{2d}x_2^{2d}x_3^{2d}.
\end{aligned}
\]
Then, by the arithmetic-geometric inequality, we see that $q_a(\bx)>0$ for all 
feasible points $\bx$. Hence, \eqref{ex::pqa} satisfies Assumption \ref{as::0} 
for all positive integers $M$ and $d$.
For $M\in\{6,8,10\}$ and $d\in\{2,3,4\}$, we solve \eqref{ex::pqa} using the minimum relaxation order $k=3d$ and present the computational results in Table \ref{tab::2}. From the table, we see that by exploiting sign symmetries, we gain speedup by $1\sim2$ magnitudes.

\begin{table}[htbp]\label{tab::2}
\caption{Computational results for \eqref{ex::pqa}.}
\centering
\begin{tabular}{ccccccccc}
\midrule[0.8pt]
\multirow{2}{*}{$M$} &&\multicolumn{3}{c}{$\sup$\eqref{eq::gmpdualsos}/Time}&
&\multicolumn{3}{c}{$\sup$\eqref{eq::gmpdualtssos}/Time} \\
		\cline{3-5} \cline{7-9}
&& $d=2$ & $d=3$& $d=4$& & $d=2$ &$d=3$& $d=4$\\
\midrule[0.4pt]
$6$&&5.00/0.31s&5.00/3.85s&5.00/36.1s &&5.00/0.06s&5.00/0.38s&5.00/1.17s \\
$8$&&7.00/0.43s&7.00/4.25s&7.00/52.4s &&7.00/0.09s&7.00/0.54s&7.00/1.84s \\
$10$&&9.00/0.54s&9.00/5.34s&9.00/126s &&9.00/0.12s&9.00/0.71s&9.00/2.43s \\
\midrule[0.8pt]
\end{tabular}
\end{table}

\end{example}

\begin{example}
Consider the following problem adapted from \cite{BHL2016}:
\begin{equation}\label{ex::css4}
\min_{\bx\in\RR^{n}}-\sum_{i=1}^N\frac{1}{\sum_{j=1}^n(x_j^2-a_{ij})^2 + c_i}\quad\textrm{ s.t. } 60-\sum_{i=1}^n(x_i^2-5)^2\ge0.
\end{equation} 
The data $\{a_{ij}\},\{c_i\}$ are given in \cite[Table 16]{ali2005numerical} with $N=30,n\in\{5,10\}$. For $n=5$, the optimum of \eqref{ex::css4} is $-10.4056$, and for $n=10$, the optimum is $-10.2088$. We present the computational results in Table \ref{tab::7}. From the table, we see that 1) the global optimum is achieved at $k=5$ for $n=5$ and is achieved at $k=4$ for $n=10$; 2) by exploiting sign symmetries, we gain speedup by $1\sim3$ magnitudes.

\begin{table}[htbp]\label{tab::7}
\caption{Computational results for \eqref{ex::css4}. The symbol `-' indicates that {\tt MOSEK} runs out of memory.}
\centering
 \resizebox{\linewidth}{!}{
\begin{tabular}{ccccccccc}
\midrule[0.8pt]
$n$ &&\multicolumn{3}{c}{$\sup$\eqref{eq::gmpdualsos}/Time}&
&\multicolumn{3}{c}{$\sup$\eqref{eq::gmpdualtssos}/Time} \\
\midrule[0.4pt]
\multirow{2}{*}{$5$}&& $k=3$ & $k=4$& $k=5$& & $k=3$ &$k=4$& $k=5$\\
\cline{3-5} \cline{7-9}
&&-49.4493/3.53s&-12.4737/27.8s&-10.4056/306s
&&-49.4493/0.19s&-12.4736/1.52s&-10.4056/5.66s \\
\midrule[0.8pt]
\multirow{2}{*}{$10$}&& $k=2$ & $k=3$& $k=4$& & $k=2$ &$k=3$& $k=4$\\
\cline{3-5} \cline{7-9}
&&-25.2408/4.68&-19.2391/1158s&- &&-25.2408/0.11&-19.2391/3.12s&-10.2088/46.8s \\
\midrule[0.8pt]
\end{tabular}}
\end{table}

\end{example}

\begin{example}
Now we carry out randomly generated problems of the form
\begin{equation}\label{ex::rand}
\min_{\bx\in\RR^n}\ -\sum_{i=1}^N\frac{1}{f_i^2+1} \quad \text{s.t.} \ \ x_1^2+\cdots+x_n^2\le 1,
\end{equation}
which are constructed as follows: 
1) randomly choose a subset of nonconstant monomials $\mathscr{M}$ from $[\bx]_d$ with prescribed probability $\xi$; 
2) randomly assign $3$ elements in $\mathscr{M}$ to each $f_i$ coupled with random coefficients in $[0,1]$. 
It is clear that the optimal value of \eqref{ex::rand} is $-N$.
We denote the set of \eqref{ex::rand} generated in such a way by \textsc{RandSRFO}$(N, n, d, \xi)$ and construct the following $12$ instances\footnote{They are available at https://wangjie212.github.io/jiewang/code.html.}:
\[
\begin{aligned}
P_{1}, P_2, P_3&\in \textsc{RandSRFO}(10, 6, 4, 0.05),\\
P_4, P_5, P_6&\in \textsc{RandSRFO}(8, 5, 5, 0.10),\\
P_7, P_8, P_9&\in \textsc{RandSRFO}(6, 8, 4, 0.03),\\
P_{10}, P_{11}, P_{12}&\in \textsc{RandSRFO}(12, 7, 5, 0.05).\\
\end{aligned}
\]
We solve those instances using the minimum relaxation order $k=d$ and present the computational results in Table \ref{tab::3}. From the table, we see that by exploiting sign symmetries, we gain speedup by $1\sim2$ magnitudes.

\begin{table}[htbp]\label{tab::3}
\caption{Computational results for \eqref{ex::rand}. The symbol `-' indicates that {\tt MOSEK} runs out of memory.}
\centering
 \resizebox{\linewidth}{!}{
	\begin{tabular}{rcccccc}
\midrule[0.8pt]
 & $P_1$& $P_2$ & $P_3$ &$P_4$ & $P_5$ & $P_6$ \\
\midrule[0.4pt]
		$\sup$\eqref{eq::gmpdualsos}/Time&-10.00/24.9s& -10.00/22.7s&-10.00/20.5s&-8.000/19.7s&-8.000/19.3s &-8.000/19.3s\\
		$\sup$\eqref{eq::gmpdualtssos}/Time&-10.00/1.97s &-10.00/0.92s&-10.00/0.89s&-8.000/1.57s&-8.000/1.41s & -8.000/1.61s\\
\midrule[0.8pt]
 & $P_7$& $P_8$ & $P_9$ &$P_{10}$ & $P_{11}$ & $P_{12}$ \\
\midrule[0.4pt]
		$\sup$\eqref{eq::gmpdualsos}/Time&-6.000/311s&-6.000/331s &-6.000/370s& - & -& -\\
		$\sup$\eqref{eq::gmpdualtssos}/Time&-6.000/3.04s&-6.000/2.18s & -6.000/2.79s&-12.00/19.7s&-12.00/19.9s&-12.00/19.8s\\
\midrule[0.8pt]
	\end{tabular}}		
\end{table}

\end{example}

\subsection{Comparison of \eqref{eq::csssdpdual} with \eqref{eq::csssdp}}
We will proceed with two examples for which both correlative sparsity and sign symmetries are present. 

\begin{example}
Consider
\begin{equation}\label{ex::css1}
\begin{aligned}
\min_{\bx\in\RR^{2N+2}}&\ \sum_{i=1}^N\frac{(x_{2i-1}^2+x_{2i}^2+x_{2i+1}^2)x_{2i-1}^2x_{2i}^2x_{2i+1}^2+x_{2i+2}^{8}}{x_{2i-1}^2x_{2i}^2x_{2i+1}^2x_{2i+2}^2}\\
\text{s.t.}\,\,\,\,\, &\ x_{2i-1}^2+x_{2i}^2+x_{2i+1}^2+x_{2i+2}^2=4,\quad i\in[N].
\end{aligned}
\end{equation}
For each $i$, the polynomial
\[
(x_{2i-1}^2+x_{2i}^2+x_{2i+1}^2-4x_{2i+2}^2)x_{2i-1}^2x_{2i}^2x_{2i+1}^2+x_{2i+2}^{8}
\]
is nonnegative with zeros $(\pm 1, \pm 1, \pm 1, \pm 1)$ \cite{MotzkinAGI}. Therefore, 
the optimum of \eqref{ex::css1} equals $4N$. It is clear that Assumption \ref{assump2} holds with $I_i=\{2i-1, 2i, 2i+1, 2i+2\}$ and $J_i=\{i\}$, $i\in[N]$. 
For $N\in\{20,40,60\}$, we solve \eqref{ex::css1} and present the computational results in Table \ref{tab::4}. It can be seen from the table that by exploiting sign symmetries, we gain extra speedup by several times.

\begin{table}[htbp]\label{tab::4}
\caption{Computational results for \eqref{ex::css1}. The symbol `*' indicates that {\tt MOSEK} encounters numerical issues.}
\centering
\begin{tabular}{ccccccc}
\midrule[0.8pt]
\multirow{2}{*}{$N$} &&\multicolumn{2}{c}{$\sup$\eqref{eq::csssdpdual}/Time}&
&\multicolumn{2}{c}{$\sup$\eqref{eq::csssdp}/Time} \\
		\cline{3-4} \cline{6-7}
&& $k=4$ & $k=5$& & $k=4$ & $k=5$\\
\midrule[0.4pt]
		20&&16.65*/1.38s& 80.00/10.6s&&19.06*/0.49s&
		80.00/0.73s\\
            40&&37.71*/3.25s& 160.0/15.8s&&37.71*/0.62s&
		160.0/1.79s\\
		60&&52.17*/6.97s& 240.0/28.9s&&57.75*/1.51s&
		240.0/3.43s\\
\midrule[0.8pt]
\end{tabular}	
\end{table}

\end{example}

\begin{example}
Consider
\begin{equation}\label{ex::css2}
\min_{\bx\in\RR^{2N+1}}\ \sum_{i=1}^N\frac{p_i(\bx)}{q_i(\bx)}\quad \text{s.t.} \ \ x_{2i-1}^2+x_{2i}^2+x_{2i+1}^2=3,\ i\in[N],
\end{equation}
where
\[
\begin{aligned}
p_i(\bx)=\,&x_{2i-1}^{6d}+x_{2i}^{6d}+x_{2i+1}^{6d}+3x_{2i-1}^{2d}x_{2i}^{2d}x_{2i+1}^{2d},\\
q_i(\bx)=\,&x_{2i-1}^{4d}x_{2i}^{2d}+x_{2i-1}^{2d}x_{2i}^{4d}+x_{2i-1}^{4d}x_{2i+1}^{2d}+x_{2i-1}^{2d}x_{2i+1}^{4d}+x_{2i}^{4d}x_{2i+1}^{2d}+x_{2i}^{2d}x_{2i+1}^{4d},
\end{aligned}
\]
for some $d\ge1$. From Example \ref{ex1}, it is easy to see that the optimum of \eqref{ex::css2} equals $N$. Moreover, Assumption \ref{assump2} holds with $I_i=\{2i-1, 2i, 2i+1\}$ and $J_i=\{i\}$, $i\in[N]$. For $N\in\{5,10,20\}$ and $d\in\{2,3,4\}$, we solve \eqref{ex::css2} using the minimum relaxation order $k=3d$ and present the computational results in Table \ref{tab::5}. Again, it can be seen from the table that by exploiting sign symmetries, we gain extra significant speedup.

\begin{table}[htbp]\label{tab::5}
\caption{Computational results for \eqref{ex::css2}.}
\centering
\begin{tabular}{ccccccccc}
\midrule[0.8pt]
\multirow{2}{*}{$N$} &&\multicolumn{3}{c}{$\sup$\eqref{eq::csssdpdual}/Time}&
&\multicolumn{3}{c}{$\sup$\eqref{eq::csssdp}/Time} \\
        \cline{3-5} \cline{7-9}
&& $d=2$ & $d=3$& $d=4$& & $d=2$ &$d=3$& $d=4$\\
\midrule[0.4pt]
        $5$&&5.000/0.28s&5.000/2.58s&5.000/23.1s &&5.000/0.06s&5.000/0.34s&5.000/1.06s \\
        $10$&&10.00/0.56s&10.00/4.35s&10.00/44.1s &&10.00/0.12s&10.00/0.80s&10.00/2.05s \\
        $20$&&20.00/1.17s&20.00/9.50s&20.00/95.8s &&20.00/0.25s&15.00/1.62s&20.00/4.40s \\
\midrule[0.8pt]
\end{tabular}
\end{table}

\end{example}

\subsection{Comparison of \eqref{eq::csssdpdual} with \eqref{eq::csssdp} and the epigraph approach}
In this subsection, we further compare \eqref{eq::csssdpdual} with the epigraph approach which translates \eqref{eq::pro} into a polynomial optimization problem:
\begin{equation}\label{sea}
\left\{
\begin{aligned}
   \inf_{c_i,\bx}&\ \sum_{i=1}^N c_i\\
   \text{s.t.}&\ p_i(\bx)-c_iq_i(\bx)=0,\quad i\in[N],\\
   &\ g_j(\bx)\ge0,\quad j\in[m].
\end{aligned}\tag{SEA}
\right.
\end{equation}
Note that \eqref{sea} inherits the correlative sparsity and sign symmetries of \eqref{eq::pro} which can be exploited to obtain sparse SDP relaxations for \eqref{sea} \cite{wang2022cs}.

\begin{example}
Consider 
\begin{equation}\label{ex::css3}
      \max_{\bx\in\RR^{N+1}}\ \sum_{i=1}^{N}\frac{1}{100(x_{i+1}^2-x_{i}^2)^2+(x_i^2-1)^2+1} 
      \quad  \text{s.t.} \ \ 16-x_{i}^2\ge0,\ i\in[N+1],
\end{equation}
which is modified from the well-known Rosenbrock problem.
Clearly, the optimum of \eqref{ex::css3} equals $N$ and Assumption \ref{assump2} holds with $I_i=J_i=\{i,i+1\}$, $i\in[N]$. 
We solve \eqref{ex::css3} for $N\in\{100,200,400\}$. It turns out that the 
SDP relaxations \eqref{eq::csssdpdual} and \eqref{eq::csssdp} attain global optimality at $k=2$ 
while the epigraph approach attains global optimality at $k=4$.
We present the computational results in Table \ref{tab::6}, from which we can see that 
the sparse SDP relaxation \eqref{eq::csssdp} is the most efficient approach.

\begin{table}[htbp]\label{tab::6}
\caption{Computational results for \eqref{ex::css3}. `EPI' indicates the lower bounds obtained by the epigraph approach.}
\centering
\begin{tabular}{ccccccc}
\midrule[0.8pt]
\multirow{2}{*}{$N$} &&\multicolumn{1}{c}{EPI/Time}&
&\multicolumn{1}{c}{$\sup$\eqref{eq::csssdpdual}/Time}&&\multicolumn{1}{c}{$\sup$\eqref{eq::csssdp}/Time}\\
\cline{3-3} \cline{5-5} \cline{7-7}
&& $k=4$& & $k=2$& &$k=2$\\
\midrule[0.4pt]
$100$&&100.0/2.30s&&100.0/0.15s&&100.0/0.06s\\
$200$&&200.0/6.94s&&200.0/0.58s&&200.0/0.21s\\
$400$&&400.0/20.8s&&400.0/1.72s&&400.0/0.97s\\
\midrule[0.8pt]
\end{tabular}
\end{table}

\end{example}

\begin{example}
Consider
\begin{equation}\label{ex::css5}
\min_{\bx\in\RR^{N+s}}\sum_{i=1}^N\frac{\sum_{j=1}^sx_{i+j-1}x_{i+j}}{1 + \sum_{j=1}^{s+1}jx_{i+j-1}^2}\quad\textrm{s.t. } 1-x_i^2\ge0,\ i\in[N+s],
\end{equation}
for some $s\ge1$. It is clear that Assumption \ref{assump2} holds with $I_i=J_i=\{i,i+1,\ldots,i+s\}$, $i\in[N]$. 
We solve \eqref{ex::css5} for $N\in\{20,40,60\}$ using the minimum relaxation order $k=3$ and present the computational results in Table \ref{tab::8}. From the table, we see that the epigraph approach
yields slightly tighter bounds and moreover, the sparse SDP relaxation \eqref{eq::csssdp} is the most efficient approach.

\begin{table}[htbp]\label{tab::8}
\caption{Computational results for \eqref{ex::css5}. `EPI' indicates the lower bounds obtained by the epigraph approach.}
\centering
\begin{tabular}{ccccccc}
\midrule[0.8pt]
\multirow{2}{*}{$N$} &&\multicolumn{1}{c}{EPI/Time}&
&\multicolumn{1}{c}{$\sup$\eqref{eq::csssdpdual}/Time}&&\multicolumn{1}{c}{$\sup$\eqref{eq::csssdp}/Time}\\
\cline{3-3} \cline{5-5} \cline{7-7}
&& $k=3$& & $k=3$& &$k=3$\\
\midrule[0.4pt]
$20$&&-4.5892/23.9s&&-4.6173/61.6s&&-4.6173/23.3s\\
$40$&&-8.7429/62.0s&&-8.8778/105s&&-8.8778/52.8s\\
$60$&&-12.901/100s&&-13.138/148s&&-13.138/70.8s\\
\midrule[0.8pt]
\end{tabular}
\end{table}

\end{example}

Finally, we would like to mention that the epigraph approach behaves very badly on the examples presented in the previous subsections.

\section{An application to maximizing sums of generalized Rayleigh quotients}\label{app}
In this section, we apply the sparse SDP relaxations to the problem of maximizing a sum of generalized Rayleigh quotients which arises from signal processing. 

In the downlink of a multi-user MIMO system, the base station can multiplex signals intended to different users on the same spectral resource. A challenging problem arising in such a scenario is the joint optimization of channel assignment (scheduling) and beamforming aimed at 
maximizing the sum-rate in each time-slot. Assuming zero-forcing linear beamforming at the base station, 
Primolevo et al. \cite{PSS2006} addressed this task by investigating a greedy method to approximately maximize the sum-rate.
At each iteration of their method a spatial channel needs to be determined to allocate a specific user, which reduces to 
maximizing a sum of generalized Rayleigh quotients of the form
\begin{equation}\label{eq::srq}
\max_{\bz\in\CC^n} \sum_{i=1}^N \frac{\bz^{\rm{H}} A_i \bz}{\bz^{\rm{H}} B_i \bz} \quad \text{s.t. } \ \|\bz\|^2=1,
\end{equation}
where $A_i, B_i\in\CC^{n\times n}$ are positive semidefinite Hermitian matrices, $\bz^{\rm{H}}$ denotes the conjugate transpose
of $\bz$, $n$ is the number of assigned spatial channels at the current iteration, and $N$ is the number of available users in the time-slot. 

When $N=2$, the real counterpart of the problem \eqref{eq::srq} also appears in the sparse Fisher discriminant analysis in pattern recognition (see \cite{DFBSR,FN,WZWCL}). The single generalized Rayleigh quotient optimization 
problem corresponds to the classical eigenvalue problem and can be solved in polynomial time \cite{PBN}.
However, when $N\ge 2$, solving \eqref{eq::srq} is much more challenging and requires sophisticated techniques (see \cite{NSX2016,WWX2018,Zhang2013} for the real case of $N=2$). 
In particular, Primolevo  et al. \cite{PSS2006} restricted $\bz$ in \eqref{eq::srq} to be columns of an identity matrix to give a suboptimal solution.

Here, we convert \eqref{eq::srq} into a real problem by specifying the real and imaginary parts of the complex variables and coefficients appearing in \eqref{eq::srq}. Let $\bz=\bx+\bi\by$, $A_i=A_{i,1}+\bi A_{i,2}$, and 
$B_i=B_{i,1}+\bi B_{i,2}$ with $\bx, \by\in\RR^n$ and $A_{i,1}, A_{i,2}, B_{i,1}, B_{i,2}\in\RR^{n\times n}$. 
Then \eqref{eq::srq} becomes
\begin{equation}\label{eq::srq2}
\max_{\bx,\by\in\RR^n} \sum_{i=1}^N \frac{\bx^{\intercal} A_{i,1}\bx -2\bx^{\intercal} A_{i,2}\by+ \by^{\intercal} A_{i,1}\by}
{\bx^{\intercal} B_{i,1}\bx -2\bx^{\intercal} B_{i,2}\by+ \by^{\intercal} B_{i,1}\by} \quad \text{s.t. }\ \Vert\bx\Vert^2+\Vert\by\Vert^2=1.
\end{equation}
In our numerical experiments, we assume that $A_i$'s are Hermitian and $B_i$'s are positive definite. 
To satisfy this condition, we generate random matrices $C_{i,1}, C_{i,2}, D_{i,1}, D_{i,2}\in\RR^{n\times n}$ with
each entry being drawn from the uniform distribution on $[0,1]$, and let
\[
A_i=(C_{i,1}+\bi C_{i,2}) + (C_{i,1}+\bi C_{i,2})^{\rm{H}},\quad B_i= (D_{i,1}+\bi D_{i,2})^{\rm{H}}(D_{i,1}+\bi D_{i,2}).
\]
We solve one instance of \eqref{eq::srq2} with \eqref{eq::gmpdualsos} and \eqref{eq::gmpdualtssos} for $(n, N)\in\{(3,20),(4,10),(5,5)\}$, and present the computational results in Table \ref{tab8}. From the table, we see that by exploiting sign symmetries, we gain around twice speedup.

\begin{table}[htbp]\label{tab8}
\caption{Computational results for \eqref{eq::srq2}.}
\centering
\begin{tabular}{ccccccccc}
\midrule[0.8pt]
\multirow{2}{*}{$(n, N)$} &&\multicolumn{2}{c}{$\sup$\eqref{eq::gmpdualsos}/Time}&
&\multicolumn{2}{c}{$\sup$\eqref{eq::gmpdualtssos}/Time} \\
        \cline{3-4} \cline{6-7}
& & $k=2$& $k=3$& &$k=2$& $k=3$\\
\midrule[0.4pt]
        $(3, 20)$&&406.7/0.61s&406.7/12.7s &&406.7/0.38s&406.7/5.91s \\
        $(4, 10)$&&26.54/1.03s&20.10/107s &&26.54/0.67s&20.10/50.5s\\
        $(5, 5)$&&20.12/2.43s&19.11/531s &&20.12/1.57s&19.11/339s\\
\midrule[0.8pt]
    \end{tabular}
\end{table}


\bibliographystyle{siamplain}
\bibliography{mrf}

\end{document}